\documentclass[A4paper,11pt]{amsart}

\usepackage[english]{babel}
\usepackage{amsmath}
\usepackage{amssymb} 
\usepackage{amsfonts}
\usepackage[all,color]{xy}
\usepackage{hyperref}
\usepackage{color}
\usepackage{lscape}
\usepackage{graphicx}
\usepackage{float}
\usepackage{enumerate}

\usepackage{tikz}
\usepackage{tikz-cd}
\usetikzlibrary{intersections}
\tikzset{
commutative diagrams/.cd,
ampersand replacement=\&]
}
\tikzstyle{braid}=[thick]
\tikzstyle{twocell}=[->,double, double equal sign distance,thick]
\tikzset{arr/.style={circle,draw,inner sep=0}}
\tikzset{empty/.style={inner sep=0pt, minimum size=0pt}}

\renewcommand{\hat}{\overleftarrow}
\renewcommand{\check}{\overrightarrow}

 \newtheorem{proposition}{Proposition}[section] 
 
 \newtheorem*{lemma*}{Lemma}
 \newtheorem{theorem}[proposition]{Theorem}
 \newtheorem{corollary}[proposition]{Corollary}
 \theoremstyle{definition}
 \newtheorem{definition}[proposition]{Definition}
 \newtheorem*{definition*}{Definition}
 \newtheorem{example}[proposition]{Example}
 \newtheorem*{example*}{Example}
 \newtheorem{remark}[proposition]{Remark}
 
\numberwithin{equation}{section}

\def\C{\mathcal C}

\def\rev{^{\mathsf{rev}}}
\def\vec{\mathsf{vec}}

\begin{document}

\title{Multiplier bialgebras in braided monoidal categories}
\author{Gabriella B\"ohm}
\address{Wigner Research Centre for Physics, H-1525 Budapest 114,
P.O.B.\ 49, Hungary} 
\email{bohm.gabriella@wigner.mta.hu}
\author{Stephen Lack}
\address{Department of Mathematics, Macquarie University NSW 2109, 
Australia}
\email{steve.lack@mq.edu.au}
\thanks{The authors gratefully acknowledge the financial support of the
  Hungarian Scientific Research Fund OTKA (grant K108384), the Australian
  Research Council Discovery Grant (DP130101969), an ARC Future Fellowship
  (FT110100385), and the Nefim Fund of Wigner RCP. Each is grateful to
  the warm hospitality of the other's institution during research visits in
  Nov-Dec 2013 (Sydney) and Sept-Oct 2014 (Budapest). Both authors are
  grateful to Jos\'e G\'omez-Torrecillas for discussions and comments.} 
\begin{abstract}
Multiplier bimonoids (or bialgebras) in arbitrary braided monoidal categories
are defined. They are shown to possess monoidal categories of comodules and
modules. These facts are explained by the structures carried by their induced
functors. 
\end{abstract} 
\date\today
\maketitle

\section{Introduction}

A {\em bimonoid} (or {\em bialgebra}) $A$ in a braided monoidal category can
equivalently be described without referring separately to the multiplication
$m:A.A \to A$ and comultiplication $d:A \to A.A$ but only the unit, the
counit, and the so-called {\em fusion morphism} 
\cite{Street:fusion} 
$$
t_1:=\big(
\xymatrix{
A.A\ar[r]^-{d.1}&
A.A.A\ar[r]^-{1.m}&
A.A}\big).
$$
In terms of these data, the axioms are given by the {\em fusion equation},
expressed by the commutativity of
\begin{equation}\label{eq:intr_fusion}
\xymatrix @C=35pt{
A^3 \ar[r]^-{1.t_1}\ar[d]_-{t_1.1}&
A^3 \ar[r]^-{b.1}&
A^3 \ar[r]^-{1.t_1}&
A^3 \ar[r]^-{b^{-1}.1}&
A^3\ar[d]^-{t_1.1}\\
A^3 \ar[rrrr]_-{1.t_1}&&&&
A^3}
\end{equation}
(where $b$ stands for the braiding) and some compatibilities of $t_1$ with the
unit and the counit. This is analyzed in \cite[Section 2]{LackStreet:skew} in
terms of {\em augmented lax tricocycloids}: a morphism $t_1$ satisfies the
fusion equation if and only if $t'=b^{-1}t_1$  satisfies the cocycle condition
$(1.t')(t'.1)(1.t')=(t'.1)(1.b^{-1})(t'.1)$; an object $A$ equipped with a
morphism $A.A\to A.A$ satisfying the cocycle condition is called a lax
tricocycloid, or an augmented lax tricocycloid if it is also unital and
counital. The existence of an antipode; that is, the Hopf condition on the
bimonoid $A$, is equivalent to the invertibility of the fusion morphism $t_1$.
 
The notion of bimonoid is self-dual; that is, symmetric under reversing the
arrows in the diagrams encoding the axioms. However, the above description in
terms of the fusion morphism $t_1$ is not self-dual. Instead, there is an
equivalent dual description in terms of
the morphism 
$$
t_2:=\big(
\xymatrix{
A.A\ar[r]^-{1.d}&
A.A.A\ar[r]^-{m.1}&
A.A}\big)
$$
which satisfies a modified fusion equation. In fact, using the braiding one
can construct two further variants
$$
\xymatrix@R=10pt{
t_3:=\big(A.A\ar[r]^-{d.1}&
A.A.A\ar[r]^-{1.b^{-1}}&
A.A.A\ar[r]^-{1.m}&
A.A\big)\\
t_4:=\big(A.A\ar[r]^-{1.d}&
A.A.A\ar[r]^-{b^{-1}.1}&
A.A.A\ar[r]^-{m.1}&
A.A\big)}
$$
of the fusion morphism. Each of the morphisms $t_1, t_2, t_3$ and $t_4$ can be
expressed in terms of each of the others with the help of the unit and the
counit of $A$.
 
Using the language of fusion morphisms, a (say, right) comodule $V$ over the
comonoid $A$ can be described equivalently in terms of a morphism $v:V.A\to
V.A$. Coassociativity and counitality of the coaction $V \to V.A$ translate to
commutativity of the respective diagrams  
\begin{equation}\label{eq:intro_comod}
\xymatrix{
V.A.A \ar[r]^-{1.t_1}\ar[d]_-{v.1}&
V.A.A \ar[r]^-{b.1}&
A.V.A \ar[r]^-{1.v}&
A.V.A \ar[r]^-{b^{-1}.1}&
V.A.A\ar[d]^-{v.1}&
V.A\ar[r]^-v\ar[rd]_-{1.e}&
V.A\ar[d]^-{1.e}\\
V.A.A \ar[rrrr]_-{1.t_1}&&&&
V.A.A&&
V}
\end{equation}
(where $e$ denotes the counit of $A$). By duality, there is a symmetric
description, in terms of $t_1$, of left modules over the monoid $A$.

A key example of a bialgebra (in the symmetric monoidal category $\vec$
of vector spaces over a field $k$) is given by the algebra $k^G$ of all
$k$-valued functions on a finite monoid $G$. In the case of functions of
finite support on an infinite monoid $G$, we still have the
multiplication and the counit of $k^G$, but not the unit or comultiplication;
instead of a bialgebra, under a mild cancellation assumption we
have a {\em multiplier bialgebra}. Motivated by this, we consider various
weakenings of the notion of bimonoid in a braided monoidal category. 

First we ask what happens to the above picture if we give up
self-duality and consider a counital, but no longer unital fusion morphism
$t_1:A.A\to A.A$? By this we mean that $t_1$ satisfies the fusion equation
\eqref{eq:intr_fusion} and is equipped with a morphism (the {\em
counit}) $e$ from $A$ to the monoidal unit; such that $(1.e)t_1=1.e$. We do
not require the existence of a unit and drop all axioms in
\cite{LackStreet:skew} that involve the unit. Then composing $t_1$ with $e.1$, 
we can still equip $A$ with a (no longer unital but still associative)
multiplication $A.A\to A$. But in the absence of a unit, there is no longer a
comultiplication $A\to A.A$. 

Our study of counital but not necessarily unital fusion morphisms 
is motivated by Van Daele's approach to {\em (regular) multiplier Hopf algebra}
\cite{VDae:MHA} (over a field) --- based on generalizations
of the fusion morphisms $t_1$ and $t_2$ (together with $t_3$ and
$t_4$ in the regular case) which are in turn counital but non-unital 
fusion morphisms (in the symmetric monoidal category of vector
spaces). In the absence of a unit for the algebra $A$, none of the morphisms
$t_1, t_2, t_3$ and $t_4$ determines the others. All of them are needed to
formulate the axioms. Although in the absence of a unit none of the 
fusion morphisms $t_1, t_2, t_3$ and $t_4$ determines a comultiplication $A\to
A.A$; thanks to their compatibility axioms any pair of them determines a
generalized comultiplication taking values in the multiplier algebra
\cite{Dauns:multiplier} of $A.A$. 

The main aim of this paper is to extend the definition of {\em (regular)
multiplier bialgebra} to any braided monoidal category. 
In formulating the definition, any reference to multipliers is completely
avoided. If applying it to the symmetric monoidal category of vector spaces,
our definition covers the notion of multiplier bialgebra in \cite[Theorem
2.11]{BoGTLC:wmba} but it is slightly more general.

We build up gradually to the full structure of regular multiplier bimonoid,
first developing in Section \ref{sec:counital_fusion_morphism} the theory of a
single (counital) fusion morphism $t_1$; then considering in Section
\ref{sec:mbm} multiplier bimonoids, involving a pair of these, such as $t_1$
and $t_2$ or $t_3$ and $t_4$; before finally considering in Section
\ref{sec:reg_mbm} regular multiplier bimonoids, which involve all four fusion
morphisms. We observed above that, unlike the situation with bimonoids, the
notion of multiplier bimonoid is not stable under passage from a braided
monoidal category to its opposite. But there are other duality principles
available in braided monoidal category: one can reverse the multiplication,
one can replace the braiding by its inverse, or one can do both. The resulting
four variants in some sense correspond to the four fusion morphisms in a
regular multiplier bimonoid. Our proofs often rely on these duality
principles. 

At each stage we study the corresponding notions of module and comodule. 
The unit of the monoid $A$ does not occur in the diagrams in
\eqref{eq:intro_comod}. Hence they can be used to define comodules also for
only counital fusion morphisms. We study this situation in Section
\ref{sec:tric_comod}, where we show that such comodules constitute a monoidal
category admitting a strict monoidal forgetful functor to the base
category. Comodules over a bimonoid (i.e. over a unital and counital fusion
morphism) constitute a monoidal category because they induce {\em bicomonads}
(that is, monoidal comonads). This is no longer true for a counital but
non-unital fusion morphism. We describe a generalization of the notion
of monoidal comonad which can be used to explain the monoidal structure in
this case. 

It is somewhat more delicate what should be a module over a counital but
non-unital fusion morphism. Replacing an associative (left)
action, we require the existence of a morphism $q:A.Q\to A.Q$ satisfying an
appropriate fusion type equation. But if $A$ has no unit, it is not immediate
what should replace the unitality of an action. One possibility is to require
that the morphism 
$$
\xymatrix{
A.Q\ar[r]^-{q}&
A.Q\ar[r]^-{e.1}&
Q}
$$
be an epimorphism. But in order to develop the theory,
we need this epimorphism to be preserved by tensoring on either side as
well as by the various fusion morphisms. So for simplicity we
suppose that it is preserved by all functors; equivalently, that it is
a split epimorphism. (When the fusion morphism is unital as well as
counital, and so corresponds to a bimonoid, this is equivalent to the usual
unitality condition for an action.) In Section \ref{sec:tric_module} we
investigate additional assumptions on a counital fusion morphism, weaker than 
unitality, under which such modules constitute a monoidal category
admitting a strict monoidal forgetful functor to the base category. 
Again, modules over a bimonoid (i.e. over a unital and counital fusion
morphism) constitute a monoidal category because they induce
{\em bimonads} (that is, opmonoidal monads). This is no longer true for a
counital but non-unital fusion morphism. We describe a generalization of
the notion of opmonoidal monad which can be used to explain the monoidal
structure in this case. 

Assuming {\em regularity} of a multiplier bimonoid in a braided monoidal
category, its comodules and modules are defined in the respective Sections
\ref{sec:comodule} and \ref{sec:module} as objects carrying compatible
(co)module structures over appropriate pairs of the fusion morphisms
$t_1, t_2, t_3$ and $t_4$. They are shown to constitute monoidal
categories. This is explained by the structure carried by their induced
functors. 
The definition of (co)modules in terms of {\em pairs} of the fusion morphisms
$t_1, t_2, t_3$ and $t_4$ goes back to \cite{VDaZha:corep_I}. This approach
becomes particularly important in the generalization \cite{Bohm:wmba_comod} to
{\em weak} multiplier bialgebras \cite{BoGTLC:wmba}, when both (co)actions are
needed to equip any (co)module with the structure of a bimodule over the
so-called base algebra.

One could also define {\em multiplier Hopf monoids} in a braided monoidal
category as multiplier bimonoids whose constituent fusion morphisms $t_1$ and
$t_2$ are isomorphisms. In the symmetric monoidal category of vector spaces
this is known to be equivalent to the existence of an {\em antipode} taking
values in the multiplier algebra \cite{VDae:MHA}. In our general setting,
however, the notion of multipliers is not available. The abstract categorical
treatment of multipliers requires the additional assumption that our braided
monoidal category is in fact {\em closed}. We plan to expound this
construction, and the resulting theory of multiplier Hopf monoids, in a
subsequent paper \cite{BL:braided_mha}. 

\subsection*{Notation}

Throughout, $\C$ is a braided monoidal category (unless otherwise stated). We
denote the monoidal product by $.$; the monoidal unit by $I$; and the braiding
by $b$. For the monoidal product of several copies of the same object also the
power notation is used: $A.A=A^2$. Composition of morphisms is denoted by
juxtaposition and the identity morphism (at any object) is denoted by $1$. We
do not assume that the monoidal structure is strict but --- relying on
coherence --- usually we omit explicitly denoting the associativity and unit
isomorphisms. 
For a braided monoidal category $\C$, we denote by $\overline \C$ the same
monoidal category $\C$ with the inverse braiding $(b_{Y,X})^{-1}:X.Y\to Y.X$. 
The {\em reverse} $\C\rev$ of $\C$ means the same category $\C$ with the
opposite monoidal product $(X,Y) \mapsto Y.X$ (thus the same monoidal unit
$I$) and the braiding $b_{Y,X}:Y.X\to X.Y$. The braided monoidal categories
$\overline \C\rev$ and $\overline {\C\rev}$ clearly coincide.

\section{Counital fusion morphisms}\label{sec:tricoc}

Motivated by our definition in Section \ref{sec:multiplier_bimonoid} of
multiplier bimonoid in a braided monoidal category, in this section we study
counital fusion morphisms possibly without unit. To any such creature we
associate a monoidal category of appropriately defined comodules. Under
further assumptions --- about certain morphisms being epimorphic --- we
associate to it a second monoidal category of suitable modules. The
monoidality of these categories is explained by the structure of the functors
induced by a counital fusion morphism.

\subsection{Definition and properties}\label{sec:counital_fusion_morphism}

Lax tricocycloids --- corresponding bijectively to fusion morphisms ---
were defined in \cite{LackStreet:skew}; where they were equipped both with a
counit and a unit. We need the following weakening.

\begin{definition}\label{def:counital_tricocycloid}
A {\em counital fusion morphism} in a braided monoidal category $\C$ is given
by a pair of morphisms $t:A^2\to A^2$ (called a {\em fusion morphism}) and
$e:A\to I$ (called the {\em counit}) such that the 
following diagrams commute. 
\begin{equation}\label{eq:counital_tricoc}
\xymatrix{
A^3 \ar[r]^-{1.t}\ar[d]_-{t.1}&
A^3 \ar[r]^-{b.1}&
A^3 \ar[r]^-{1.t}&
A^3 \ar[r]^-{b^{-1}.1}&
A^3\ar[d]^-{t.1}&
A^2\ar[r]^-t\ar[rd]_-{1.e}&
A^2\ar[d]^-{1.e}\\
A^3 \ar[rrrr]_-{1.t}&&&&
A^3&&
A}
\end{equation}
We refer to the first condition as the {\em fusion equation} and to the
second one as the {\em counitality condition}.
\end{definition}

A class of examples, albeit unital ones, comes from bimonoids in braided
monoidal categories:

\begin{example}\label{ex:bimonoid}
Consider a bimonoid $A$ in a braided monoidal category $\C$. Denote its monoid
structure by $(m:A^2\to A, u:I\to A)$ and denote the comonoid structure by
$(d:A\to A^2, e:A\to I)$. Then
$$
t:=\big(
\xymatrix{
A^2\ar[r]^-{d.1}&
A^3\ar[r]^-{1.m}&
A^2}\big)
$$
is a counital fusion morphism. This is easiest to see using string diagrams;
when we draw

\begin{equation*}
\begin{tikzpicture} 
\path (0,1) node[empty,name=s1t] {}
(1,1) node[empty,name=s2t]  {} 
(0.5,0.5) node[empty,name=m] {} 
(0.5,0) node[empty,name=sb] {}
(1.5,0.5) node{} ;
\path (-0.5,0.5) node {$m=$};
\draw[braid] (s1t) to[out=270,in=180] (m) to[out=0,in=270] (s2t);
\draw[braid] (m) to (sb) ;
\end{tikzpicture}~~~~
\quad
\begin{tikzpicture} 
\path (0,0) node[empty,name=s1t] {}
(1,0) node[empty,name=s2t]  {} 
(0.5,0.5) node[empty,name=m] {} 
(0.5,1) node[empty,name=sb] {}
(1.5,0.5) node{} ;
\path (-0.5,0.5) node {$d=$};
\draw[braid] (s1t) to[out=90,in=180] (m) to[out=0,in=90] (s2t);
\draw[braid] (m) to (sb) ;
\end{tikzpicture}~~~~
\quad
\begin{tikzpicture} 
\path (0,1) node[empty,name=s1t] {}
(0,0.3) node[empty,name=e] {$\circ$} 
(0.0,0) node[empty,name=sb] {}
(1.0,0.5) node{} ;
\path (-0.5,0.5) node {$e=$};
\draw[braid] (s1t) to (e);
\end{tikzpicture}~~~~
\quad
\begin{tikzpicture} 
\path (0,1) node[empty,name=s1t] {}
(1,1) node[empty,name=s2t] {}
(0,0) node[empty,name=s1b] {}
(1,0) node[empty,name=s2b] {}
(1.5,0.5) node{} ;
\path (-0.5,0.5) node {$b=$};
\draw[braid] (s2t) to (s1b);
\fill[white] (0.5,0.5) circle (0.1);
\draw[braid] (s1t) to (s2b);
\end{tikzpicture} . 
\end{equation*}
The fusion equation follows by the calculation 
\medskip

\begin{equation*}
\begin{tikzpicture}[every node/.style={empty}] 
    \path (0.25,2.0) node[name=s1t] {}
(0.75,2.0) node[name=s2t] {}
(1.5,2.0) node[name=s3t] {}
(-0.25,0) node[name=s1b] {} 
(0.25,0) node[name=s2b] {}
(1.0,0) node[name=s3b] {};
\path (0.75,1.6) node[name=d1] {}
(0.25,1.2) node[name=d2] {} 
(0.0,0.8) node[name=d3] {};
\path (1.25,1.2) node[name=m1] {}
(1.0,0.4) node[name=m3] {}
(0.25,0.4) node[name=m2] {} ;
\draw[braid] (s2t) to[out=270,in=90] (d1); 
\draw[braid] (s3t) to[out=270,in=0] (m1) to[out=180,in=0] (d1) ;
\draw[braid] (m2) to (s2b);
\draw[braid] (m3) to (s3b);
\draw[braid] (s1t) to (d2);
\path[braid, name path=y] (m1) to[out=270,in=0] (m3) to[out=180,in=0] (d2) to[out=180,in=90] (d3);
\draw[braid, name path=x] (d1) to[out=180,in=0] (m2) to[out=180,in=0] (d3) to[out=180,in=90] (s1b);
\fill[white, name intersections={of=x and y}] (intersection-1) circle(0.1);
\draw[braid] (m1) to[out=270,in=0] (m3) to[out=180,in=0] (d2) to[out=180,in=90] (d3);
\draw (1.8,0.75) node {$=$};
  \end{tikzpicture}
  \begin{tikzpicture}[every node/.style={empty}] 
    \path (0.25,2.0) node[name=s1t] {}
(1.0,2.0) node[name=s2t] {}
(1.5,2.0) node[name=s3t] {}
(-0.0,0) node[name=s1b] {} 
(0.5,0) node[name=s2b] {}
(1.25,0) node[name=s3b] {};
\path (0.25,1.6) node[name=d1] {}
(0.5,1.2) node[name=d2] {} 
(1.0,1.2) node[name=d3] {};
\path (0.5,0.8) node[name=m1] {}
(1.0,0.8) node[name=m2] {}
(1.25,0.4) node[name=m3] {} ;
\draw[braid] (s1t) to (d1) to[out=180,in=90] (s1b);
\draw[braid] (s2t) to[out=270,in=90] (d3); 
\draw[braid] (s3t) to[out=270,in=0] (m3) to[out=180,in=270] (m2) to[out=0,in=0] (d3) ;
\draw[braid] (m1) to (s2b);
\draw[braid] (m3) to (s3b);
\path[braid, name path=y] (m2) to[out=180,in=0] (d2);
\draw[braid, name path=x] (d3) to[out=180,in=0] (m1) to[out=180,in=180] (d2) to[out=90,in=0] (d1);
\fill[white, name intersections={of=x and y}] (intersection-1) circle(0.1);
\draw[braid] (m2) to[out=180,in=0] (d2);
\draw (2.0,0.75) node {$=$};
  \end{tikzpicture}
  \begin{tikzpicture}[baseline=(s1b),every node/.style={empty}] 
    \path (0.25,2.0) node[name=s1t] {}
(1.0,2.0) node[name=s2t] {}
(1.5,2.0) node[name=s3t] {}
(-0.0,0) node[name=s1b] {} 
(0.5,0) node[name=s2b] {}
(1.25,0) node[name=s3b] {};
\path (0.25,1.6) node[name=d1] {}
(0.75,0.8) node[name=d2] {} ;
\path (0.75,1.2) node[name=m1] {}
(1.25,0.4) node[name=m2] {} 
 (0.5,1.4) node[name=v1] {} 
 (1,0.6) node[name=v2] {};
\draw[braid] (s1t) to (d1) to[out=180,in=90] (s1b);
\draw[braid] (s2t) to[out=270,in=0] (m1) to[out=180,in=270] (v1) to[out=90,in=0] (d1); 
\draw[braid] (s3t) to[out=270,in=0] (m2) to[out=180,in=270] (v2) to[out=90,in=0] (d2) to[out=180,in=270] (s2b) ;
\draw[braid] (m1) to (d2);
\draw[braid] (m2) to (s3b);
\end{tikzpicture}
\end{equation*}

\noindent
where in the first equality we used naturality of the braiding, associativity
of the multiplication, and coassociativity of the comultiplication. In the
second equality we used multiplicativity of the comultiplication. The
counitality condition follows by the calculation 

\begin{equation*}
\begin{tikzpicture}[every node/.style={empty}] 
    \path (0.25,1.5) node[name=s1t] {}
(1.0,1.5) node[name=s2t] {}
 (0,0) node[name=sb] {} 
(0.25,1.0) node[name=d] {} 
(0.5,0.75) node[name=v] {}
(0.75,0.5) node[name=m] {};
\draw (0.75,0.2) node[name=e] {$\circ$};
\draw[braid] (s2t) to[out=270,in=0] (m) to[out=180,in=270] (v) to[out=90,in=0] (d) to[out=180,in=90] (sb);
\draw[braid] (s1t) to (d) ;
\draw (e) to (m);
\draw (1.5,0.75) node{$=$};
\draw (1.75,0.75) node{};
\end{tikzpicture}
\begin{tikzpicture}[every node/.style={empty}] 
    \path (0.25,1.5) node[name=s1t] {}
(0.75,1.5) node[name=s2t] {}
 (0,0) node[name=sb] {} 
(0.25,1.0) node[name=d] {} ;
\draw (0.5,0.2) node[name=e] {$\circ$};
\draw (0.75,0.2) node[name=e2] {$\circ$};
\draw[braid] (e) to[out=90,in=0] (d) to[out=180,in=90]  (sb);
\draw[braid] (s1t) to (d) ;
\draw (e2) to (s2t);
\draw (1.25,0.75) node{$=$};
\draw (1.75,0.75) node{};
\end{tikzpicture}
\begin{tikzpicture}[every node/.style={empty}] 
    \path (0.25,1.5) node[name=s1t] {}
(0.75,1.5) node[name=s2t] {}
 (0.25,0) node[name=sb] {} ;
\draw (0.75,0.2) node[name=e] {$\circ$};
\draw[braid] (s1t) to (sb) ;
\draw (e) to (s2t);
\end{tikzpicture}
\end{equation*}
\medskip

\noindent
where the first equality uses multiplicativity of the counit and the second
one counitality of the comultiplication. 
\end{example}

Further examples, no longer unital, come from multiplier bialgebras over a
field:

\begin{example}\label{ex:multiplier_bialg}
By a {\em multiplier bialgebra} over a field $k$ we mean the structure
in \cite[Theorem 2.11]{BoGTLC:wmba}. Based on
\cite[Theorem 1.2]{Bohm:wmba_comod} and \cite[Proposition 2.6]{BoGTLC:wmba},
it can be described as follows. A multiplier bialgebra is given by a vector
space $A$ equipped with an associative but not necessarily unital
multiplication $m:A^2\to A$; which is required to be surjective and
non-degenerate in the sense that both maps $A\to \mathsf{End(A)}$, $a\mapsto
m(a.-)$ and $a\mapsto m(-.a)$ are injective. Furthermore, the existence of
linear maps $t_1,t_2:A^2\to A^2$ and $e:A\to k$ is required such that the
following axioms hold (where $b$ denotes the symmetry in $\mathsf{vec}$; of
course its components satisfy $b^2=1$).
\begin{itemize}
\item[{(a)}] $t_1(m.1)=(m.1)(b.1)(1.t_1)(b.1)(1.t_1)$; equivalently,\\
$t_2(1.m)=(1.m)(1.b)(t_2.1)(1.b)(t_2.1)$.
\item[{(b)}] Both maps $(m.1)(b.1)(1.t_1)$ and $(1.m)(1.b)(t_2.1)$ are
surjective.
\item[{(c)}] $em=e.e$.
\item[{(d)}] $(t_2.1)(1.t_1)=(1.t_1)(t_2.1)$.
\item[{(e)}] $(e.1)t_1=m=(1.e)t_2$.
\end{itemize}
We claim that $t_1$ is then a counital fusion morphism in
$\mathsf{vec}$. Indeed, by (e) and (d), 
\begin{equation}\label{eq:m_t_1-2}
(m.1)(1.t_1)=(1.m)(t_2.1). 
\end{equation}
Postcomposing \eqref{eq:m_t_1-2} with $1.e$, we obtain by (c) and (e)
$$
m[1.(1.e)t_1]=
(1.e)t_2.e=
m(1.1.e),
$$ 
from which we conclude by the non-degeneracy of $m$ that the counitality
condition $(1.e)t_1=1.e$ holds. Furthermore, 
\begin{equation}\label{eq:short_fusion}
\begin{tikzpicture}
  \path (1,2) node[arr,name=t1u] {$t_1$} 
(1,1) node[arr,name=t1d]  {$t_1$} 
(0,0.3) node[arr,name=t1l] {$t_1$} 
(-0.5,-0.2) node[empty,name=m] {};
\path (1.1,2.7) node[empty,name=s3t] {} 
(0.5,2.7) node[empty,name=s2t] {} 
(0,2.7) node[empty,name=s1t] {} 
(-0.5,-0.5) node[empty,name=s1b] {} 
(0.5,-0.5) node[empty,name=s2b] {} 
(1,-0.5) node[empty,name=s3b] {} 
(-0.5,2.7) node[empty,name=s0t] {};
\draw[braid] (s3t) to[out=315,in=45] (t1u);
\draw[braid] (s2t) to[out=270,in=135] (t1u);
\draw[braid] (t1u) to[out=315,in=45] (t1d);
\draw[braid] (t1d) to[out=315,in=45] (s3b);
\path[braid,name path=s4] (s1t) to[out=270,in=135] (t1d);
\draw[braid,name path=s5] (t1u) to[out=225,in=0] (t1l);
\fill[white, name intersections={of=s4 and s5}] (intersection-1) circle(0.1);
\path[braid,name path=s6] (t1d) to[out=210,in=0] (0,0.8) [out=180,in=180] (t1l);
\fill[white, name intersections={of=s6 and s5}] (intersection-1) circle(0.1);
\draw[braid,name path=s4] (s1t) to[out=270,in=135] (t1d);
\draw[braid,name path=s6] (t1d) to[out=210,in=0] (0,0.8) to[out=180,in=180] (t1l);
\draw[braid] (m) to (s1b);
\draw[braid] (s0t) to[out=270,in=180] (m) to[out=0,in=225] (t1l);
\draw[braid] (t1l) to[out=315,in=90] (s2b);
\draw (2,1.35) node {$\stackrel{\eqref{eq:m_t_1-2}}=$};
\end{tikzpicture}
\begin{tikzpicture}
  \path (1,2) node[arr,name=t1u] {$t_1$} 
(1,1) node[arr,name=t1d]  {$t_1$} 
(-0.5,0.3) node[arr,name=t2] {$t_2$} 
(-0.0,-0.2) node[empty,name=m] {};
\path (1.1,2.7) node[empty,name=s3t] {} 
(0.5,2.7) node[empty,name=s2t] {} 
(0,2.7) node[empty,name=s1t] {} 
(-0.5,-0.5) node[empty,name=s1b] {} 
(0.0,-0.5) node[empty,name=s2b] {} 
(1,-0.5) node[empty,name=s3b] {} 
(-0.5,2.7) node[empty,name=s0t] {};
\draw[braid] (s3t) to[out=315,in=45] (t1u);
\draw[braid] (s2t) to[out=270,in=135] (t1u);
\draw[braid] (t1u) to[out=315,in=45] (t1d);
\draw[braid] (t1d) to[out=315,in=45] (s3b);
\path[braid,name path=s4] (s1t) to[out=270,in=135] (t1d);
\draw[braid,name path=s5] (t1u) to[out=225,in=0] (m);
\fill[white, name intersections={of=s4 and s5}] (intersection-1) circle(0.1);
\path[braid,name path=s6] (t1d) to[out=210,in=45] (t2);
\fill[white, name intersections={of=s6 and s5}] (intersection-1) circle(0.1);
\draw[braid] (s1t) to[out=270,in=135] (t1d);
\draw[braid] (t1d) to[out=225,in=45] (t2);
\draw[braid] (m) to (s2b);
\draw[braid] (s0t) to[out=270,in=135] (t2);
\draw[braid] (t2) to[out=315,in=180] (m);
\draw[braid] (t2) to[out=225,in=135] (s1b);
\draw (2,1.35) node {$\stackrel{(d)}=$};
\end{tikzpicture}
\begin{tikzpicture}
  \path (0.75,2) node[arr,name=t1u] {$t_1$} 
(0.75,1) node[arr,name=t1d]  {$t_1$} 
(-0.25,2) node[arr,name=t2] {$t_2$} 
(-0.0,-0.2) node[empty,name=m] {};
\path (1.0,2.7) node[empty,name=s3t] {} 
(0.5,2.7) node[empty,name=s2t] {} 
(0,2.7) node[empty,name=s1t] {} 
(-0.5,-0.5) node[empty,name=s1b] {} 
(0.0,-0.5) node[empty,name=s2b] {} 
(1,-0.5) node[empty,name=s3b] {} 
(-0.5,2.7) node[empty,name=s0t] {};
\draw[braid] (s3t) to[out=315,in=45] (t1u);
\draw[braid] (s2t) to[out=270,in=135] (t1u);
\draw[braid] (t1u) to[out=315,in=45] (t1d);
\draw[braid] (t1d) to[out=300,in=80] (s3b);
\draw[braid] (s0t) to[out=225,in=135] (t2);
\draw[braid] (s1t) to[out=315,in=45] (t2);
\draw[braid] (t2) to[out=260,in=100] (s1b);
\path[braid,name path=s4] (t2) to[out=315,in=135] (t1d);
\path[braid,name path=s6] (t1d) to[out=210,in=180] (m);
\draw[braid,name path=s5] (t1u) to[out=225,in=0] (m);
\fill[white, name intersections={of=s4 and s5}] (intersection-1) circle(0.1);
\fill[white, name intersections={of=s6 and s5}] (intersection-1) circle(0.1);
\draw[braid] (t2) to[out=315,in=135] (t1d);
\draw[braid] (t1d) to[out=210,in=180] (m);
\draw[braid] (m) to (s2b);
\draw (1.6,1.35) node {$\stackrel{(a)}=$};
\end{tikzpicture}
\begin{tikzpicture}
  \path (0.75,1.0) node[arr,name=t1] {$t_1$} 
(0.5,1.5) node[arr,name=m]  {} 
(-0.25,2) node[arr,name=t2] {$t_2$}; 
\path (1.0,2.7) node[empty,name=s3t] {} 
(0.5,2.7) node[empty,name=s2t] {} 
(0,2.7) node[empty,name=s1t] {} 
(-0.5,-0.5) node[empty,name=s1b] {} 
(0.0,-0.5) node[empty,name=s2b] {} 
(1,-0.5) node[empty,name=s3b] {} 
(-0.5,2.7) node[empty,name=s0t] {};
\draw[braid] (s3t) to[out=270,in=45] (t1);
\draw[braid] (s2t) to[out=270,in=0] (m) to[out=180,in=315] (t2);
\draw[braid] (s0t) to[out=225,in=135] (t2);
\draw[braid] (s1t) to[out=315,in=45] (t2);
\draw[braid] (t2) to[out=260,in=100] (s1b);
\draw[braid] (m) to[out=270,in=135] (t1);
\draw[braid] (t1) to[out=225,in=90]  (s2b);
\draw[braid] (t1) to[out=280,in=80] (s3b);
\draw (1.6,1.35) node {$\stackrel{\eqref{eq:m_t_1-2}}=$};
\end{tikzpicture}
\begin{tikzpicture}
  \path (0.75,1.0) node[arr,name=t1d] {$t_1$} 
(0.25,2) node[arr,name=t1u] {$t_1$}
(-0.25,0.5) node[arr,name=m]  {} ;
\path (1.0,2.7) node[empty,name=s3t] {} 
(0.5,2.7) node[empty,name=s2t] {} 
(0,2.7) node[empty,name=s1t] {} 
(-0.25,-0.5) node[empty,name=s1b] {} 
(0.5,-0.5) node[empty,name=s2b] {} 
(1,-0.5) node[empty,name=s3b] {} 
(-0.5,2.7) node[empty,name=s0t] {};
\draw[braid] (s3t) to[out=270,in=45] (t1d);
\draw[braid] (s2t) to[out=270,in=45] (t1u); 
\draw[braid] (s1t) to[out=315,in=135] (t1u);
\draw[braid] (s0t) to[out=270,in=180] (m) to[out=0,in=225] (t1u);
\draw[braid] (m) to[out=270,in=90] (s1b);
\draw[braid] (t1u) to[out=315,in=135] (t1d);
\draw[braid] (t1d) to[out=250,in=100]  (s2b);
\draw[braid] (t1d) to[out=290,in=80] (s3b);
\end{tikzpicture}
\end{equation}
\noindent
from which we conclude by the non-degeneracy of $m$ that the fusion equation
$(t_1.1)t_1^{13}(1.t_1)=(1.t_1)(t_1.1)$ holds. Dually, $t_2$ is a counital
fusion morphism in $\vec\rev$.
\end{example}

\begin{proposition}\label{prop:multiplication}
For a counital fusion morphism $(t:A^2\to A^2,e:A\to I)$ in a braided monoidal
category $\C$, the following assertions hold.
\begin{itemize}
\item[{(1)}] There is an associative multiplication
$
m:=\big(
\xymatrix@C=20pt{
A^2\ar[r]^-t&
A^2\ar[r]^-{e.1}&
A}\big).
$
\item[{(2)}] $(1.m)(t.1)=t(1.m)$.
\item[{(3)}] $(m.1)(b^{-1}.1)(1.t)(b.1)(1.t)=t(m.1)$.
\item[{(4)}] $em=e.e$.
\end{itemize}
\end{proposition}

\begin{proof} 
In the following diagram, the top region commutes by the fusion equation and
the triangular region commutes by the counitality condition. The bottom left
region commutes by functoriality of the monoidal product and coherence of the
braiding. 
$$
\xymatrix{
A^3\ar[rrrr]^-{t.1}\ar[d]_-{1.t}&&&&
A^3
\ar[d]^-{1.t}&
\\
A^3\ar[r]^-{b.1}\ar[d]_-{1.e.1}&
A^3\ar[r]^-{1.t}&
A^3\ar[r]^-{b^{-1}.1}&
A^3\ar[r]^-{t.1}\ar[rd]_-{1.e.1}&
A^3
\ar[d]^-{1.e.1}&
\\
A^2\ar[rrrr]_-t&&&&
A^2}
$$
This proves part (2) and postcomposing both paths by $e.1$ we obtain part (1)
by functoriality of the monoidal product.
In order to prove (3), postcompose both sides of the fusion equation by
$e.1.1$ and use functoriality of the monoidal product. For (4), postcompose
both sides of the counitality condition by $e$ and use
functoriality of the monoidal product.
\end{proof}

\subsection{Comodules and right multiplier bicomonads} 
\label{sec:tric_comod}

Our approach to the study of comodules over a counital fusion morphism is
based on the use of the following notion.

\begin{definition}\label{def:left_multi_bicomonad}
A {\em right multiplier bicomonad} on a monoidal category $\C$ is
a functor \hbox{$G:\C\to \C$} equipped with natural transformations $\check
G_2:GX.GY\to G(GX.Y)$ and \hbox{$\varepsilon:GX\to X$} such that the diagrams
\begin{equation}\label{eq:G_check}
\xymatrix@C=20pt{
GX.GY.GZ\ar[r]^-{1.\check G_2}\ar[d]_-{\check G_2.1}&
GX.G(GY.Z)\ar[r]^-{\check G_2}&
G(GX.GY.Z)\ar[d]^-{G(\check G_2.1)}&
GX.GY\ar[r]^-{\check G_2}\ar[rd]_-{1.\varepsilon}&
G(GX.Y)\ar[d]^-\varepsilon\\
G(GX.Y).GZ
\ar[rr]_-{\check G_2}&&
G(G(GX.Y).Z)&&
GX.Y}
\end{equation}
commute for any objects $X,Y,Z$. 

A {\em left multiplier bicomonad} on $\C$ is a right multiplier bicomonad on
the reverse monoidal category $\C\rev$. 
\end{definition}

To a right multiplier bicomonad $G$, we can associate a natural transformation
$$ 
\xymatrix{
GX.GY\ar[r]^-{\check G_2}&
G(GX.Y)\ar[r]^-{G(\varepsilon.1)}&
G(X.Y).}
$$
It satisfies the same associativity condition as the binary part of a monoidal
functor $G$ but it has no nullary part; thus it makes $G$ what might be called
a {\em semimonoidal functor}. A right multiplier bicomonad is indeed 
a generalization of monoidal comonad (also known as {\em `bicomonad'}) as the
following example shows:

\begin{example}\label{ex:G_check_from_bicomonad}
Let $(G,\delta,\varepsilon)$ be a monoidal comonad on a monoidal category
$\C$; that is, assume the existence of a monoidal structure $(G_2:G(-).G(-)
\to G(-.-), G_0:I\to GI)$ on the functor $G$, and the monoidality of the
coassociative comultiplication $\delta:G\to G^2$ and of the counit
$\varepsilon:G\to 1$. Consider 
$$
\check G_2:=\big(
\xymatrix{
GX.GY\ar[r]^-{\delta.1}&
G^2X.GY\ar[r]^-{G_2}&
G(GX.Y)}\big)\ .
$$
Then the diagrams in \eqref{eq:G_check} commute by the coassociativity of
$\delta$, the associativity of $G_2$, and by the monoidality of $\delta$ on
one hand; and by the monoidality of $\varepsilon$ and the counitality of
$\delta$ on the other hand. 
\end{example}

Further examples are provided by the functors induced by counital fusion
morphisms:

\begin{example}\label{ex:G_check_from_t}
Let $(t:A^2\to A^2,e:A\to I)$ be a counital fusion morphism in a braided
monoidal category $\C$. Consider the functor $G:=(-).A:\C\to \C$ with the
natural transformations $\varepsilon:=1.e$ and 
$$
\check G_2:=\big(
\xymatrix@C=35pt{
X.A.Y.A\ar[r]^-{1.b.1}&
X.Y.A^2\ar[r]^-{1.1.t}&
X.Y.A^2\ar[r]^-{1.b^{-1}.1}&
X.A.Y.A}\big).
$$
Then the diagrams in \eqref{eq:G_check} commute by the fusion equation and by
the counitality condition on $t$, respectively.
\end{example}

\begin{definition}\label{def:G_check_comodule}
Consider a right multiplier bicomonad $G$ on a monoidal category $\C$. A
$G$-{\em comodule} is an object $V$ together with a natural transformation
$\check v:V.G(-)\to G(V.-)$ rendering commutative, for any objects $Y$ and
$Z$, the diagrams 
\begin{equation}\label{eq:v_check}
\xymatrix@C=20pt{
V.GY.GZ\ar[r]^-{1.\check G_2}\ar[d]_-{\check v.1}&
V.G(GY.Z)\ar[r]^-{\check v}&
G(V.GY.Z)\ar[d]^-{G(\check v.1)}&
V.GY\ar[r]^-{\check v}\ar[rd]_-{1.\varepsilon}&
G(V.Y)\ar[d]^-\varepsilon\\
G(V.Y).GZ\ar[rr]_-{\check G_2}&&
G(G(V.Y).Z)&&
V.Y.}
\end{equation}
A {\em morphism} of $G$-comodules is a morphism $f:V\to W$ such that
$G(f.1)\check v =\check w (f.1)$.

Comodules over a left multiplier bicomonad are defined as comodules over the
corresponding right multiplier bicomonad on $\C\rev$. 
\end{definition}

These comodules behave well with respect to the monoidal structure of the base
category: 

\begin{theorem}\label{thm:mon_cat_of_check_G_comodules}
Consider a right multiplier bicomonad $G$ on a monoidal category $\C$. Then
the $G$-comodules of Definition~\ref{def:G_check_comodule} and their morphisms
constitute a monoidal category such that the evident forgetful functor to
$\C$ is strict monoidal.
\end{theorem}

\begin{proof}
The monoidal unit $I$ is a $G$-comodule via the composite
$$
\xymatrix{
I.G(-)\ar[r]^-\cong &
G \ar[r]^-\cong &
G(I.-)}
$$
of the unit isomorphisms. For $G$-comodules $\check v:V.G(-)\to G(V.-)$ and
$\check w:W.G(-)\to G(W.-)$, also $V.W$ is a $G$-comodule via 
$$
\xymatrix{
V.W.G(-)\ar[r]^-{1.\check w}&
V.G(W.-)\ar[r]^-{\check v}&
G(V.W.-).}
$$
The monoidal product of $G$-comodule morphisms, as well as the unit and
associativity isomorphisms, are evidently morphisms of $G$-comodules.
\end{proof}

\begin{example}\label{ex:EM_comodule}
We claim that applying Definition \ref{def:G_check_comodule} to the right
multiplier bicomonad $G$ in Example \ref{ex:G_check_from_bicomonad}, we 
obtain a category of comodules which is equivalent (in fact, isomorphic) to
the usual Eilenberg-Moore category of comodules (or coalgebras). Therefore
Theorem \ref{thm:mon_cat_of_check_G_comodules} generalizes the well-known
result \cite{McCrudden} that the monoidal structure of the base category lifts
to the Eilenberg-Moore category of a monoidal comonad. 

Consider the category of $G$-comodules in Definition
\ref{def:G_check_comodule} and the forgetful functor to $\C$ sending
$(V,\check v)$ to $V$. We shall show that for a right multiplier bicomonad $G$ 
as in Example \ref{ex:G_check_from_bicomonad}, the forgetful functor $U$ has a
right adjoint $F$ such that $G=UF$ as comonads. Indeed, let $F$ take an object
$X$ to the $G$-comodule $(GX,\check G_2)$. The counit of the adjunction is
$\varepsilon: UF=G \to \C$ and the unit $\eta$ evaluated at a $G$-comodule 
$(V,\check v)$ is 
$$
\xymatrix{
V\ar[r]^-{1.G_0}&
V.GI\ar[r]^-{\check v}&
GV.}
$$
This is a morphism of $G$-comodules, in the sense of Definition
\ref{def:G_check_comodule}, by commutativity of the diagram
$$
\xymatrix{
V.GX\ar[r]^-{1.G_0.1}\ar[d]^-{1.G_0.1}\ar@/_3pc/@{=}[dd]&
V.GI.GX\ar[r]^-{\check v.1}\ar[d]^-{1.\delta.1}&
GV.GX\ar[dd]^-{\delta.1}\\
V.GI.GX\ar[r]^-{1.GG_0.1}\ar[d]^-{1.G_2}&
V.G^2I.GX\ar[d]^-{1.G_2}\\
V.GX\ar[r]^-{1.G(G_0.1)}\ar[d]_-{\check v}&
V.G(GI.X)\ar[d]^-{\check v}&
G^2V.GX\ar[d]^-{G_2}\\
G(V.X)\ar[r]_-{G(1.G_0.1)}&
G(V.GI.X)\ar[r]_-{G(\check v.1)}&
G(GV.X)}
$$
for any object $X$ of $\C$.
The region on the right commutes by the fusion equation on $\check v$; the
top-left square commutes by the monoidality of $\delta$, and the regions below
it commute by the naturality of $G_2$ and $\check v$, respectively. The
leftmost region commutes by the unitality of the monoidal structure
$(G_2,G_0)$. 
Evaluating $\eta$ at a $G$-comodule of the form $GX=(GX,\check G_2)$, we
obtain $G_2(\delta.1)(1.G_0)=\delta$ (where the equality follows by the
functoriality of the monoidal product and the unitality of $(G_2,G_0)$). Hence
the first triangle condition on the adjunction $U\dashv F$ follows by the
counitality of $\delta$. The other triangle condition holds by the counitality
of $\check v$ and the monoidality of $\varepsilon$. In order to see that the 
evidently conservative left adjoint functor $U$ is comonadic, we need to prove 
that it creates $U$-absolute equalizers. Suppose then that 
$f$ and $g$ are morphisms $(V,\check{v})\to(W,\check{w})$ and that 
$$\xymatrix{
Z \ar[r]^{h} & V \ar@<.3ex>[r]^{f} \ar@<-.3ex>[r]_{g} & W }$$
is an absolute equalizer in $\C$. Then in the solid part of the diagram 
$$\xymatrix{
Z.GX \ar[r]^{h.1} 
\ar@{-->}[d]_{\check{z}} & 
V.GX \ar@<.3ex>[r]^{f.1} \ar@<-.3ex>[r]_{g.1} 
\ar[d]_{\check{v}} & W.GX \ar[d]^{\check{w}} \\
G(Z.X) \ar[r]_{G(h.1)} & 
G(V.X) \ar@<.3ex>[r]^{G(f.1)} \ar@<-.3ex>[r]_{G(g.1)} & 
G(W.X)}$$
the rows are equalizers, and the parallel pairs commute serially with the
verticals, thus there is a unique induced morphism $\check{z}$ as
in the dashed part of the diagram. 
The axioms for $(Z,\check{z})$ to be a comodule follow easily from the
corresponding axioms for $(V,\check{v})$ and the fact that $h$ is an absolute
monomorphism. The morphism $h$ preserves the comodule structure by
construction, and the universal property of $(Z,\check{z})$ follows from the
universal property of $Z$ along with the fact that $G(h.1)$ is a
monomorphism. 

Explicitly, the inverse of the comparison functor -- from the category
of $G$-comodules in Definition \ref{def:G_check_comodule} to the category of
Eilenberg-Moore $G$-comodules -- sends an Eilenberg-Moore comodule $v:V\to GV$
to the equalizer of the comodule morphisms $\delta,Gv:(GV,\check G_2) \to
(G^2V,\check G_2)$; that is, to the comodule
$$
\xymatrix{
V.G(-) \ar[r]^-{v.1} &
GV.G(-) \ar[r]^-{G_2} &
G(V.-).} 
$$ 
\end{example}

Further examples of the situation in Definition \ref{def:G_check_comodule} are
provided by the following.

\begin{definition}\label{def:tric_comodule}
Consider a counital fusion morphism $(t:A^2\to A^2,e:A\to I)$ in a braided
monoidal category $\C$. A {\em comodule} over it is an object $V$ together
with a morphism $v:V.A\to V.A$ in $\C$ rendering commutative the following
diagrams.
\begin{equation}\label{eq:comod}
\xymatrix{
V.A^2 \ar[r]^-{1.t}\ar[d]_-{v.1}&
V.A^2 \ar[r]^-{b.1}&
A.V.A \ar[r]^-{1.v}&
A.V.A \ar[r]^-{b^{-1}.1}&
V.A^2\ar[d]^-{v.1}&
V.A\ar[r]^-v\ar[rd]_-{1.e}&
V.A\ar[d]^-{1.e}\\
V.A^2 \ar[rrrr]_-{1.t}&&&&
V.A^2&&
V.}
\end{equation}
A {\em morphism} of comodules is a morphism $f:V\to W$ in $\C$ such that 
$w(f.1)=(f.1)v$. 
\end{definition}

\begin{remark}
A comodule $v:V.A\to V.A$ as in Definition \ref{def:tric_comodule} induces a
comodule over the right multiplier bicomonad $(-).A$ in Example
\ref{ex:G_check_from_t} by putting
$$
\xymatrix{
V.X.A\ar[r]^-{b.1}&
X.V.A\ar[r]^-{1.v}&
X.V.A\ar[r]^-{b^{-1}.1}&
V.X.A.}
$$
This is in fact the object-part of a fully faithful functor from the category
of comodules for $A$ to the category of comodules for the induced comonad $G$.
It's not hard to see that the functor is injective on
objects; in many concrete cases, such as $\C=\vec$, it is an isomorphism of 
categories.
\end{remark}

Hence from Theorem \ref{thm:mon_cat_of_check_G_comodules} we have the
following.

\begin{corollary} \label{cor:mon_cat_of_tric_comodules}
Consider a counital fusion morphism $(t:A^2\to A^2,e:A\to I)$ in a braided
monoidal category $\C$. Its category of comodules, as in Definition
\ref{def:tric_comodule}, is monoidal in such a way that the evident forgetful
functor to $\C$ is strict monoidal. 
\end{corollary}

\proof
It suffices to observe that the full subcategory of $G$-comodules consisting
of the comodules for $A$ is closed under the monoidal structure. Explicitly, 
the monoidal unit $I$ is a comodule via the identity morphism $A\to A$ and the
monoidal product of comodules $v:V.A\to V.A$ and $w:W.A \to W.A$ is a comodule
via
$$
\xymatrix@C=30pt{
V.W.A\ar[r]^-{1.w}&
V.W.A\ar[r]^-{b.1}&
W.V.A\ar[r]^-{1.v}&
W.V.A\ar[r]^-{b^{-1}.1}&
V.W.A.}
$$
\endproof

\subsection{Modules and left multiplier bimonads}
\label{sec:tric_module}

In studying modules over a counital fusion morphism, we rely on the following
notion.

\begin{definition}\label{def:right_multi_bimonad}
A {\em left multiplier bimonad} on a monoidal category $\C$ is a
functor $T:\C\to \C$ equipped with a natural transformation $\hat T_2: T(X.TY)
\to TX.TY$ and a morphism $T_0:TI\to I$ such that the diagrams
\begin{equation}\label{eq:T_hat}
\xymatrix@C=20pt{
T(X.T(Y.TZ))\ar[r]^-{T(1.\hat T_2)}\ar[d]_-{\hat T_2}&
T(X.TY.TZ) \ar[r]^-{\hat T_2}&
T(X.TY).TZ\ar[d]^-{\hat T_2.1}&
T(X.TI)\ar[r]^-{\hat T_2}\ar[rd]_(.36){T(1.T_0)}&
TX.TI\ar[d]^(.56){1.T_0}\\
TX.T(Y.TZ)\ar[rr]_-{1.\hat T_2}&&
TX.TY.TZ&&
TX}
\end{equation}
commute for any objects $X,Y,Z$. 

A {\em right multiplier bimonad} on $\C$ is a left multiplier bimonad on the
reverse monoidal category $\C\rev$.
\end{definition}

To a left multiplier bimonad $T$, we can associate a natural transformation
\begin{equation}\label{eq:mu_on_T}
\mu:=\big(
\xymatrix{
T^2X\ar[r]^-{\hat T_2}&
TI.TX\ar[r]^-{T_0.1}&
TX}\big).
\end{equation}
This yields an associative but non-unital multiplication. Left 
multiplier bimonads are indeed a generalization of opmonoidal monads (also
known as {\em `bimonads'}) as the following example shows: 

\begin{example}\label{ex:T_hat_from_bimonad}
Let $(T,\mu,\eta)$ be an opmonoidal monad on a monoidal category $\C$; that
is, assume the existence of an opmonoidal structure $(T_2:
T(-.-)\to T(-).T(-), T_0:TI\to I)$ and the opmonoidality of the associative 
multiplication $\mu:T^2\to T$ and of the unit $\eta:1\to T$.
Consider 
$$
\hat T_2:=\big(
\xymatrix{
T(X.TY)\ar[r]^-{T_2}&
TX.T^2Y\ar[r]^-{1.\mu}&
TX.TY
}\big) .
$$
Then the diagrams in \eqref{eq:T_hat} commute by the associativity of
$\mu$, the coassociativity of $T_2$, and the opmonoidality of $\mu$ on one
hand, and by the opmonoidality of $\mu$ and the counitality of $T_2$ on the
other hand.
\end{example}

Further examples are provided by the functors induced by counital fusion
morphisms: 

\begin{example}\label{ex:T_hat_from_t}
Let $(t:A^2\to A^2, e:A\to I)$ be a counital fusion morphism in a braided
monoidal category $\C$. Consider the functor $T:=A.(-):\C\to \C$ with the
morphism $T_0:=e:A\to I$ and the natural transformation 
$$
\hat T_2:=\big(
\xymatrix@C=35pt{
A.X.A.Y\ar[r]^-{1.b^{-1}.1}&
A.A.X.Y\ar[r]^-{t.1.1}&
A.A.X.Y\ar[r]^-{1.b.1}&
A.X.A.Y}\big).
$$
Then the diagrams in \eqref{eq:T_hat} commute by the fusion equation and
by the counitality condition on $t$, respectively. In this example, the
multiplication \eqref{eq:mu_on_T} is induced by the multiplication in
Proposition \ref{prop:multiplication}~(1).
\end{example}

\begin{definition}\label{def:T_hat_module}
Consider a left multiplier bimonad $T$ on a monoidal category $\C$. A $T$-{\em
module} is an object $Q$ together with a natural transformation $\hat q:
T(-.Q)\to T(-).Q$ such that the diagram
$$
\xymatrix{
T(X.T(Y.Q)) \ar[r]^-{T(1.\hat q)}\ar[d]_-{\hat T_2}&
T(X.TY.Q)\ar[r]^-{\hat q}&
T(X.TY).Q\ar[d]^-{\hat T_2.1}\\
TX.T(Y.Q)\ar[rr]_-{1.\hat q}&&
TX.TY.Q}
$$
commutes for any objects $X$ and $Y$; and the morphism 
$
\xymatrix@C=20pt{
TQ\ar[r]^-{\hat q}&
TI.Q\ar[r]^-{T_0.1}&
Q}
$
--- which is in fact an associative action with respect to the multiplication
\eqref{eq:mu_on_T} --- is a split epimorphism. (Note that an associative
action $TQ\to Q$ by a {\em unital} monad $T$ is a split epimorphism if and
only if it is unital.) 

A {\em morphism} of modules from $(Q,\hat q)$ to $(R,\hat r)$ is a morphism
$Q\to R$ in $\C$ satisfying the evident compatibility condition. 

A module over a right multiplier bimonad on $\C$ is defined as a
module over the corresponding left multiplier bimonad on $\C\rev$. 
\end{definition}

\begin{theorem}\label{thm:mon_cat_of_hat_T_modules}
Consider a left multiplier bimonad $T$ on a monoidal category $\C$. Assume
that both $T_0$ and the morphisms
$$
\xymatrix{
T(TX.TY)\ar[r]^-{\hat T_2}&
T^2X.TY\ar[r]^-{\hat T_2.1}&
TI.TX.TY\ar[r]^-{T_0.1.1}&
TX.TY,}
$$
for any objects $X$ and $Y$, are split epimorphisms. Then the $T$-modules
of Definition \ref{def:T_hat_module} constitute a monoidal category such that
the evident forgetful functor to $\C$ is strict monoidal. 
\end{theorem}

\begin{proof}
The monoidal unit $I$ is a $T$-module via the composite of the unit constraints
$$
\xymatrix{
T(-.I)\ar[r]^-\cong&
T\ar[r]^-\cong& 
T(-).I.
}
$$
Indeed, the fusion equation holds by coherence (and naturality of $\hat
T_2$ and of the unit isomorphisms) and $T_0$ is a split epimorphism by
assumption. For $T$-modules $\hat q:T(-.Q)\to T(-).Q$ and $\hat p:T(-.P)\to
T(-).P$, we consider the candidate module structure 
$$
\xymatrix{
T(-.P.Q)\ar[r]^-{\hat q}&
T(-.P).Q\ar[r]^-{\hat p.1}&
T(-).P.Q.}
$$
It satisfies the fusion equation since both $\hat q$ and $\hat p$ do: for
any objects $X$ and $Y$ the diagram
$$
\xymatrix{
T(X.T(Y.P.Q))\ar[r]^-{T(1.\hat q)}\ar[ddd]_-{\hat T_2}&
T(X.T(Y.P).Q)\ar[r]^-{T(1.\hat p.1)}\ar[d]^-{\hat q}&
T(X.TY.P.Q)\ar[d]^-{\hat q}\\
&
T(X.T(Y.P)).Q\ar[r]^-{T(1.\hat p).1}\ar[dd]^-{\hat T_2.1}&
T(X.TY.P).Q\ar[d]^-{\hat p.1}\\
&&
T(X.TY).P.Q\ar[d]^-{\hat T_2.1.1}\\
TX.T(Y.P.Q)\ar[r]_-{1.\hat q}&
TX.T(Y.P).Q\ar[r]_-{1.\hat p.1}&
TX.TY.P.Q}
$$
commutes. In the diagram
$$
\xymatrix{
T(TP.TQ)\ar[r]^-{T(\hat p.1)}\ar[d]_-{\hat T_2}&
T(TI.P.TQ)\ar[r]^-{T(T_0.1.1)}&
T(P.TQ)\ar[r]^-{T(1.\hat q)}\ar[d]^-{\hat T_2}\ar@{}[rrd]|-{(\ast)}&
T(P.TI.Q)\ar[r]^-{T(1.T_0.1)}&
T(P.Q)\ar[d]^-{\hat q}\\
T^2P.TQ\ar[r]^-{T\hat p.1}\ar[d]_-{\hat T_2.1}\ar@{}[rrd]|-{(\ast)}&
T(TI.P).TQ\ar[r]^-{T(T_0.1).1}&
TP.TQ\ar[r]^-{1.\hat q}\ar[d]^-{\hat p.1}&
TP.TI.Q\ar[r]^-{1.T_0.1}&
TP.Q\ar[d]^-{\hat p.1}\\
TI.TP.TQ\ar[r]^-{1.\hat p.1}\ar[d]_-{T_0.1.1}&
TI.TI.P.TQ\ar[r]^-{1.T_0.1.1}&
TI.P.TQ\ar[d]^-{T_0.1.1}&&
TI.P.Q\ar[d]^-{T_0.1.1}\\
TP.TQ\ar[r]_-{\hat p.1}&
TI.P.TQ\ar[r]_-{T_0.1.1}&
P.TQ\ar[r]_-{1.\hat q}&
P.TI.Q\ar[r]_-{1.T_0.1}&
P.Q}
$$
the left column and the bottom row are split epimorphisms by assumption. Hence
also the top-right path is a split epimorphism proving that so is the morphism
in the right column. In this diagram the unlabelled regions commute by
naturality of $\hat T_2$ and functoriality of the monoidal product. The
regions marked by $(\ast)$ commute since for any $T$-module $\hat q:T(-.Q)\to
T(-).Q$ and any object $X$, the diagram 
$$
\xymatrix@C=35pt{
T(X.TQ)\ar[r]^-{T(1.\hat q)}\ar[dd]_-{\hat T_2}&
T(X.TI.Q)\ar[r]^-{T(1.T_0.1)}\ar[d]^-{\hat q}&
T(X.Q)\ar[d]^-{\hat q}\\
&
T(X.TI).Q\ar[r]^-{T(1.T_0).1}\ar[d]^-{\hat T_2.1}&
TX.Q\ar@{=}[d]\\
TX.TQ\ar[r]_-{1.\hat q}&
TX.TI.Q\ar[r]_-{1.T_0.1}&
TX.Q}
$$
commutes by naturality of $\hat q$, by the counitality condition in
\eqref{eq:T_hat}, and by the fusion equation on $\hat q$. 
The monoidal product of $T$-module morphisms, as well as the unit and
associativity isomorphisms are evidently morphisms of $T$-modules.
\end{proof}

\begin{example}\label{ex:EM_module}
We claim that applying Definition \ref{def:T_hat_module} to the functor
$T$ in Example \ref{ex:T_hat_from_bimonad}, we obtain a category of
modules which is equivalent (in fact, isomorphic) to the usual
Eilenberg-Moore category of modules (or algebras). Therefore Theorem
\ref{thm:mon_cat_of_hat_T_modules} generalizes the well-known result
\cite{McCrudden} that the monoidal structure of the base category lifts to the
Eilenberg-Moore category of an opmonoidal monad.

The reasoning is similar to Example \ref{ex:EM_comodule}. Consider the
category of $T$-modules in Definition \ref{def:T_hat_module} and the forgetful
functor $U$ from it to $\C$. For a left multiplier bimonad $T$ as
in Example \ref{ex:T_hat_from_bimonad}, $U$ possesses a left adjoint $F$ such
that $UF=T$ as monads: $F$ takes an object $X$ to the $T$-module $(TX,\hat
T_2)$. (It obeys the fusion equation by assumption and \eqref{eq:mu_on_T} is an
epimorphism split by $\eta$.) The unit of the adjunction is $\eta:\C\to T=UF$
and the counit, evaluated at a $T$-module $(Q,\hat q)$, is $q:=(T_0.1)\hat
q:TQ\to Q$. (This is a morphism of $T$-modules, in the sense of Definition
\ref{def:T_hat_module}, by the fusion equation on $\hat q$, the counitality of
$\hat T_2=(1.\mu)T_2$, and the naturality of $\hat q$.) Then the counit at an
object $FX=(TX,\hat T_2)$ is equal to
$$
\big(\xymatrix{
T^2X\ar[r]^-{T_2}&
TI.T^2X\ar[r]^-{1.\mu}&
TI.TX\ar[r]^-{T_0.1}&
TX}\big)=\mu
$$
hence the first triangle condition follows by the unitality of $\mu$. Since
for a $T$-module $(Q,\hat q)$, $q$ is an associative action which is epi by
assumption, the other triangle condition follows by the naturality of $\eta$,
the associativity of $q$, and the unitality of $\mu$ again:
$$
q\eta q=
q(Tq)\eta=
q\mu\eta=q.
$$
In order to see that the obviously conservative right adjoint functor $U$ is
monadic, we need to prove that it creates $U$-absolute coequalizers.
Suppose then that $f$ and $g$ are morphisms $(Q,\hat{q})\to(R,\hat{r})$, and
that 
$$\xymatrix{
Q \ar@<.3ex>[r]^{f} \ar@<-.3ex>[r]_{g} & R \ar[r]^{h} & S }$$
is an absolute coequalizer in $\C$. Then in the solid part of the diagram 
$$\xymatrix{
T(X.Q) \ar@<.3ex>[r]^{T(1.f)} \ar@<-.3ex>[r]_{T(1.g)} \ar[d]_{\hat{q}} & 
T(X.R) \ar[r]^{T(1.h)} \ar[d]_{\hat{r}} & 
T(X.S) \ar@{-->}[d]^{\hat{s}} \\
TX.Q \ar@<.3ex>[r]^{1.f} \ar@<-.3ex>[r]_{1.g} & TX.R \ar[r]^{1.h} & TX.S }$$
the rows are coequalizers, and the parallel pairs commute serially with the
verticals, thus there is a unique induced morphism $\hat{s}$ as
in the dashed part of the diagram. The axioms for $(S,\hat{s})$ to be a module
follow easily from the corresponding axioms for $(R,\hat{r})$ and the fact
that $h$ is an absolute (thus split) epimorphism. The morphism $h$ preserves
the module structure by construction, and the universal property of
$(S,\hat{s})$ follows from the universal property of $S$ and the fact that
$T(1.h)$ is an epimorphism.
Explicitly, the inverse of the comparison functor -- from the category
of $T$-modules in Definition \ref{def:T_hat_module} to the category of
Eilenberg-Moore $T$-modules -- sends an Eilenberg-Moore module $q:TQ\to Q$
to the coequalizer of the $T$-module morphisms $\mu,Tq:
(T^2Q,\hat T_2)\to (TQ,\hat T_2)$; that is, to the $T$-module
$$
\xymatrix{
T(-.Q)\ar[r]^-{T_2}&
T(-).TQ\ar[r]^-{1.q}&
T(-).Q.}
$$ 
\end{example}

Further examples of the situation in Definition \ref{def:T_hat_module} are
provided by the following.

\begin{definition}\label{def:tric_module}
Consider a counital fusion morphism $(t:A^2\to A^2,e:A\to I)$ in a braided
monoidal category $\C$. A {\em module} over it is an object $Q$ together with
a morphism $q:A.Q\to A.Q$ in $\C$ such that the diagram
\begin{equation}\label{eq:mod}
\xymatrix{
A^2.Q\ar[r]^-{1.q}\ar[d]_-{t.1}&
A^2.Q\ar[r]^-{b.1}&
A^2.Q\ar[r]^-{1.q}&
A^2.Q\ar[r]^-{b^{-1}.1}&
A^2.Q\ar[d]^-{t.1}\\
A^2.Q\ar[rrrr]_-{1.q}&&&&
A^2.Q}
\end{equation}
commutes and
$
\xymatrix@C=15pt{
A.Q\ar[r]^-q&
A.Q\ar[r]^-{e.1}&
Q}
$
is a split epimorphism.
A {\em morphism} of modules is a morphism $f:Q\to R$ in $\C$ such that
$(1.f)q=r(1.f)$. 
\end{definition}

\begin{example}\label{ex:tric_module_from_mba}
Consider a multiplier bialgebra $A$ over a field. By Example
\ref{ex:multiplier_bialg}, there is associated to it a counital fusion
morphism $t_1$ in $\mathsf{vec}$. We claim that there is an isomorphism $\Phi$
from the category of modules in Definition \ref{def:tric_module} to the
following category. The objects are vector spaces $Q$ equipped with an
associative $A$-action $q:A.Q\to Q$ which is in addition a surjective  
map. The morphisms are the linear maps which commute with the actions.

Suppose then that $q_1\colon A.Q\to A.Q$ is a module as in
Definition~\ref{def:tric_module}. Composing the fusion equation~\eqref{eq:mod} 
with $1.e.1$, and writing $q$ for the surjection $(e.1)q_1$, we see that the
diagram
\begin{equation}\label{eq:construct_q1}
\xymatrix{
A^2.Q\ar[r]^-{t_1.1}\ar[d]_-{1.q}&
A^2.Q\ar[d]^-{1.q}\\
A.Q\ar[r]_-{q_1}&
A.Q\ }
\end{equation}
commutes. Composing this with $e.1$, we see that $q$ is associative. This
defines the value on objects of a functor $\Phi$, which acts as the identity
on morphisms. 

Since $1.q$ is, like $q$, surjective, we can recover $q_1$ from $q$, and so
$\Phi$ is injective on objects. Using surjectivity of $1.q$ once again, one
deduces that $\Phi$ is full. 

Thus it remains only to show that $\Phi$ is surjective on objects. To do this,
let $q\colon A.Q\to Q$ be a surjective associative action, and use the fact
that in $\mathsf{vec}$ every surjective map --- so in particular the map
$1.q$ --- is the cokernel of its kernel. So to deduce the existence of a map
$q_1$ making the diagram in \eqref{eq:construct_q1} commute, it will suffice
to show that $(1.q)(t_1.1)$ vanishes on the kernel of $1.q$. 

By \eqref{eq:m_t_1-2} and by the associativity of $q$, 
$$
(1.q)(m.1.1)(1.t_1.1)= 
(1.q)(1.m.1)(t_2.1.1)=
(1.q)(1.1.q)(t_2.1.1)=
(1.q)(t_2.1)(1.1.q)
$$
so that $(m.1)(1.1.q)(1.t_1.1)$ vanishes on
$\mathsf{ker}(1.1.q)=A.\mathsf{ker}(1.q)$. By the non-degeneracy of
$m$ this proves that $(1.q)(t_1.1)$ vanishes on $\mathsf{ker}(1.q)$, so that
$q_1$ is indeed well-defined. Using the fact that $q$ is surjective, the
fusion equation \eqref{eq:mod} for $q_1$ follows by 
\begin{eqnarray*}
(t_1.1)q_1^{13}(1.q_1)(1.1.q)&=&
(t_1.1)q_1^{13}(1.1.q)(1.t_1.1)=
(t_1.1)(1.1.q)(t_1^{13}.1)(1.t_1.1)\\
&=&(1.1.q)(t_1.1.1)t_1^{13}(1.t_1.1)=
(1.1.q)(1.t_1.1)(t_1.1.1)\\
&=&(1.q_1)(1.1.q)(t_1.1.1)=
(1.q_1)(t_1.1)(1.1.q).
\end{eqnarray*}
In the fourth equality we used the fusion equation on $t_1$. Again by the
surjectivity of $q$,
\begin{eqnarray*}
(e.1)q_1(1.q)=
(e.1)(1.q)(t_1.1)=
q(m.1)=
q(1.q)
\end{eqnarray*}
implies $(e.1)q_1=q$ which is surjective (hence a split epimorphism) by
assumption. In the penultimate equality we used the functoriality of 
the monoidal product and axiom (e) in Example \ref{ex:multiplier_bialg}, and
in the last equality we used the associativity of $q$.
\end{example}

\begin{remark}
A module $q:A.Q\to A.Q$ as in Definition \ref{def:tric_module} induces a
module over the left multiplier bimonad $A.(-)$ in Example
\ref{ex:T_hat_from_t} by putting 
$$
\xymatrix{
A.X.Q\ar[r]^-{1.b^{-1}}&
A.Q.X\ar[r]^-{q.1}&
A.Q.X\ar[r]^-{1.b}&
A.X.Q.}
$$
Once again, this defines the object-part of a fully faithful injective
functor from the category of $A$-modules to the category of $T$-modules for
the induced monad $T$; once again this will be an isomorphism in many
concrete cases, such as $\C=\vec$. 
\end{remark}

Hence from Theorem \ref{thm:mon_cat_of_hat_T_modules} we have the following.

\begin{corollary} \label{cor:mon_cat_of_tric_modules}
Consider a counital fusion morphism $(t:A^2\to A^2,e:A\to I)$ in a braided
monoidal category $\C$. Assume that $e$ and
$$
\xymatrix{
A^3\ar[r]^-{1.t}&
A^3\ar[r]^-{b^{-1}.1}&
A^3\ar[r]^-{m.1}&
A^2}
$$
are split epimorphisms. Its category of modules, as in Definition
\ref{def:tric_module}, is monoidal in such a way that the evident forgetful
functor to $\C$ is strict monoidal. 
\end{corollary}

The monoidal unit $I$ is a module via the identity morphism $A\to A$, and the
monoidal product of modules $q:A.Q\to A.Q$ and $p:A.P\to A.P$ is a module via
$$
\xymatrix{
A.P.Q\ar[r]^-{b.1}&
P.A.Q\ar[r]^-{1.q}&
P.A.Q\ar[r]^-{b^{-1}.1}&
A.P.Q\ar[r]^-{p.1}&
A.P.Q.}
$$
In particular, in view of Example \ref{ex:tric_module_from_mba} we conclude
that for a multiplier bialgebra $A$ over a field, the category of associative
$A$-modules with a surjective action is monoidal via the tensor product of
vector spaces.

\section{Multiplier bimonoids in braided monoidal categories}
\label{sec:multiplier_bimonoid} 

In this section we define the central object studied in the paper ---
{\em multiplier bimonoids} in a braided monoidal category ---
using the theory of counital fusion morphisms developed in the previous
section. We discuss their further properties like {\em regularity} and
{\em non-degeneracy} of the multiplication (in Proposition
\ref{prop:multiplication}~(1)). We show how the motivating examples ---
bimonoids in braided monoidal categories and multiplier bialgebras over a
field --- are covered by our definition. 

\subsection{Multiplier bimonoid}\label{sec:mbm} 
A multiplier bimonoid is defined in terms of a compatible pair of
fusion morphisms: 

\begin{definition} \label{def:mbm}
A {\em multiplier bimonoid} in a braided monoidal category $\C$ consists of a
fusion morphism $t_1:A^2\to A^2$ in $\C$ and a fusion morphism $t_2:A^2\to
A^2$ in $\C\rev$ possessing a common counit $e:A\to I$ such that the following
diagrams commute. 
$$
\xymatrix{
A^3\ar[r]^-{t_2.1}\ar[d]_-{1.t_1}&
A^3\ar[d]^-{1.t_1}&&
A^2\ar[r]^-{t_1}\ar[d]_-{t_2}&
A^2\ar[d]^-{e.1}\\
A^3\ar[r]_-{t_2.1}&
A^3&&
A^2\ar[r]_-{1.e}&
A}
$$
\end{definition}

The second diagram in Definition \ref{def:mbm} expresses the requirement that
the multiplications, corresponding as in Proposition
\ref{prop:multiplication}~(1) to the counital fusion morphisms $t_1$ and
$t_2$, coincide. In the axioms in Definition \ref{def:mbm} the roles of $t_1$
and $t_2$ are symmetric: $(t_1,t_2,e)$ is a multiplier bimonoid in $\C$ if and
only if $(t_2,t_1,e)$ is a multiplier bimonoid in $\C\rev$.

As the name suggests, this is a common generalization of bimonoids in braided
monoidal categories and of multiplier bialgebras over a field:

\begin{example} \label{ex:mbm_from_bimonoid}
A bimonoid $A$ in a braided monoidal category $\C$ determines a multiplier
bimonoid in $\C$ by
$$
t_1:=\big(
\xymatrix{
A^2\ar[r]^-{d.1}&
A^3\ar[r]^-{1.m}&
A^2
}\big)\qquad
t_2:=\big(
\xymatrix{
A^2\ar[r]^-{1.d}&
A^3\ar[r]^-{m.1}&
A^2
}\big) .
$$
Indeed, $t_1$ is a counital fusion morphism in $\C$ by Example
\ref{ex:bimonoid} and $t_2$ is a counital fusion morphism in $\C\rev$
by symmetry. The first diagram in Definition \ref{def:mbm} commutes by
the functoriality of the monoidal product and the coassociativity of the
comultiplication. The second diagram in Definition \ref{def:mbm} commutes by
the functoriality of the monoidal product and the counitality of the
comultiplication.
\end{example}

\begin{example} \label{ex:mbm_from_mba}
For a multiplier bialgebra over a field, the maps $t_1$ and $t_2$ in Example
\ref{ex:multiplier_bialg} constitute a multiplier bimonoid in
$\mathsf{vec}$; see axioms (d) and (e) in Example \ref{ex:multiplier_bialg}. 
\end{example}

Also a certain converse holds:

\begin{proposition} \label{prop:mba_from_mbm}
Consider a multiplier bimonoid $(t_1,t_2:A^2\to A^2,e:A\to I)$ in
$\mathsf{vec}$ and denote $m:=(e.1)t_1=(1.e)t_2$ (it is an associative
multiplication by Proposition \ref{prop:multiplication}~(1)). If
\begin{itemize}
\item $m$, $(m.1)(b.1)(1.t_1)$ and $(1.m)(1.b)(t_2.1)$ are surjective and 
\item $m$ is non-degenerate,
\end{itemize}
then $A$ is a multiplier bialgebra (in the sense recalled in Example
\ref{ex:multiplier_bialg}).
\end{proposition}

\begin{proof}
Axiom (c) in Example \ref{ex:multiplier_bialg} holds by
Proposition \ref{prop:multiplication}~(4).
Axiom (a) in Example \ref{ex:multiplier_bialg} follows by postcomposing the
fusion equation on $t_1$ by $e.1.1$ (or postcomposing the fusion equation on
$t_2$ by $1.1.e$). 
\end{proof}
 
Examples of multiplier bimonoids in $\mathsf{vec}$ which are not multiplier
bialgebras, however, can be obtained as linear spans of semigroups $S$ (that is,
non-unital monoids $S$ in $\mathsf{set}$). In this case the fusion maps are
given by 
$$
t_1:a. b \mapsto a.ab\qquad \qquad
t_2:a.b \mapsto ab. b
$$
on the linear basis $\{a.b\ |\ a,b\in S\}$ --- where juxtaposition
denotes the multiplication in $S$ --- and the counit is the linear map
sending any element of $S$ to the unit element of the base field. These need
not be multiplier bialgebras because the multiplication need not
be non-degenerate, and the surjectivity condition (b) of
Example~\ref{ex:multiplier_bialg} need not hold.  

Non-degeneracy of the multiplication can be formulated also in our context:

\begin{definition}\label{def:non-deg}
Consider a morphism $m:A^2\to A$ in a monoidal category. We say that $m$ is
{\em non-degenerate} if for any objects $X$ and $Y$, both maps 
\begin{eqnarray*}
&\C(X,Y.A)\to \C(X.A,Y.A),\qquad
&f\mapsto \big(
\xymatrix@C=18pt{
X.A\ar[r]^-{f.1}&
Y.A^2\ar[r]^-{1.m}&
Y.A\,}\big)\\
&\C(X,A.Y)\to \C(A.X,A.Y),\qquad
&g\mapsto \big(
\xymatrix@C=18pt{
A.X\ar[r]^-{1.g}&
A^2.Y\ar[r]^-{m.1}&
A.Y}\big)
\end{eqnarray*}
are injective.
\end{definition}

Clearly, if $m$ is a unital multiplication then it is non-degenerate.
If the multiplication of a multiplier bimonoid is non-degenerate, then some of
the axioms in Definition \ref{def:mbm} become redundant. 

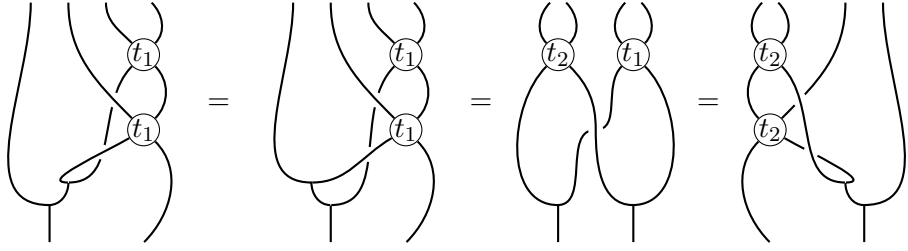
\begin{remark}\label{rmk:short}
Consider morphisms $t_1,t_2:A^2\to A^2$ and $e:A\to I$ in a braided
monoidal category $\C$ such that $(t_2.1)(1.t_1)=(1.t_1)(t_2.1)$ and the
morphisms $(e.1)t_1$ and $(1.e)t_2$ are equal and non-degenerate. Observe
that, by the argument given in \eqref{eq:short_fusion},  
the fusion equation for $t_1$ follows from  the ``short'' fusion equation on
the left  in
\begin{equation}\label{eq:short-fusion}
\begin{tikzpicture}
  \path (7,2) node[arr,name=t1u] {$t_1$} (7,1) node[arr,name=t1d]  {$t_1$} (6,0.3) node[empty,name=m3] {};
\path (7.1,2.7) node[empty,name=l2s1t] {} (6.5,2.7) node[empty,name=l2s2t] {} (6,2.7) node[empty,name=l2s3t] {} (7,0) node[empty,name=l2s1b] {} (6,0) node[empty,name=l2s2b] {};
\draw[braid] (l2s1t) to[out=315,in=45] (t1u);
\draw[braid] (l2s2t) to[out=270,in=135] (t1u);
\draw[braid] (t1u) to[out=315,in=45] (t1d);
\draw[braid] (t1d) to[out=315,in=45] (l2s1b);
\path[braid,name path=s4] (l2s3t) to[out=270,in=135] (t1d);
\draw[braid,name path=s5] (t1u) to[out=225,in=0] (m3);
\fill[white, name intersections={of=s4 and s5}] (intersection-1) circle(0.1);
\path[braid,name path=s6] (t1d) to[out=210,in=180] (m3);
\fill[white, name intersections={of=s6 and s5}] (intersection-1) circle(0.1);
\draw[braid,name path=s4] (l2s3t) to[out=270,in=135] (t1d);
\draw[braid,name path=s6] (t1d) to[out=210,in=180] (m3);
\draw[braid] (m3) to (l2s2b);
\draw (8,1.35) node {$=$};
\path (9.6,1) node[arr,name=t1r] {$t_1$} 
(9.25,2) node[empty,name=m4] {};
\path (10.1,2.7) node[empty,name=r2s1t] {} 
(9.5,2.7) node[empty,name=r2s2t] {} 
(9,2.7) node[empty,name=r2s3t] {}
(10,0) node[empty,name=r2s1b] {} 
(9,0) node[empty,name=r2s2b] {};
\draw[braid] (r2s1t) to[out=250,in=45] (t1r);
\draw[braid] (r2s2t) to [out=270,in=0] (m4) to [out=180,in=270] (r2s3t);
\draw[braid] (m4) to[out=270,in=135] (t1r);
\draw[braid] (t1r) to[out=225,in=90] (r2s2b);
\draw[braid] (t1r) to[out=315,in=90] (r2s1b); 
\end{tikzpicture}
\qquad
\begin{tikzpicture}
  \path (0,2) node[arr,name=t2u] {$t_2$} (0,1) node[arr,name=t2d]  {$t_2$} (1,0.3) node[empty,name=m] {};
\path (-0.1,2.7) node[empty,name=s1t] {} (0.5,2.7) node[empty,name=s2t] {} (1,2.7) node[empty,name=s3t] {} (0,0) node[empty,name=s1b] {} (1,0) node[empty,name=s2b] {};
\draw[braid] (s1t) to[out=225,in=135] (t2u);
\draw[braid] (s2t) to[out=270,in=45] (t2u);
\draw[braid] (t2u) to[out=225,in=135] (t2d);
\draw[braid] (t2d) to[out=225,in=135] (s1b);
\draw[braid,name path=s1] (s3t) to[out=270,in=45] (t2d);
\path[braid,name path=s2] (t2u) to[out=315,in=180] (m);
\fill[white, name intersections={of=s1 and s2}] (intersection-1) circle(0.1);
\draw[braid,name path=s3] (t2d) to[out=330,in=0] (m);
\fill[white, name intersections={of=s3 and s2}] (intersection-1) circle(0.1);
\draw[braid] (t2u) to[out=315,in=180] (m);
\draw[braid] (m) to (s2b);
\draw (2,1.35) node {$=$};
\path (3.4,1) node[arr,name=t2r] {$t_2$} 
(3.75,2) node[empty,name=m2] {};
\path (2.9,2.7) node[empty,name=rs1t] {} 
(3.5,2.7) node[empty,name=rs2t] {} 
(4,2.7) node[empty,name=rs3t] {}
(3,0) node[empty,name=rs1b] {} 
(4,0) node[empty,name=rs2b] {};
\draw[braid] (rs1t) to[out=250,in=135] (t2r);
\draw[braid] (rs2t) to [out=270,in=180] (m2) to [out=0,in=270] (rs3t);
\draw[braid] (m2) to[out=270,in=45] (t2r);
\draw[braid] (t2r) to[out=315,in=90] (rs2b);
\draw[braid] (t2r) to[out=225,in=90] (rs1b); 
\end{tikzpicture}~~~~ .
\end{equation}

On the other hand the short fusion equation follows from the fusion equation
by composing with counit on the first string, thus the two equations are
equivalent. Dually, the fusion equation for $t_2$ is equivalent to its short
version appearing on the right in \eqref{eq:short-fusion}.
\end{remark}

\begin{proposition} \label{prop:mb-nd}
Consider morphisms $t_1,t_2:A^2\to A^2$ and $e:A\to I$ in a braided monoidal
category $\C$ such that $(e.1)t_1=(1.e)t_2$ and
$(t_2.1)(1.t_1)=(1.t_1)(t_2.1)$. Assume that $m:=(e.1)t_1=(1.e)t_2$ is
non-degenerate.
 
(1) The following assertions are equivalent to each other.
\begin{itemize}
\item[{(i)}] $t_1$ is a fusion morphism in $\C$.
\item[{(ii)}] $t_2$ is a fusion morphism in $\C\rev$.
\end{itemize}

(2) The following assertions are also equivalent to each other.
\begin{itemize}
\item[{(i)}] $(1.e)t_1=1.e$.
\item[{(ii)}] $(e.1)t_2=e.1$.
\item[{(iii)}] $em=e.e$.
\end{itemize}

\noindent
The datum $(t_1,t_2,e)$ is a multiplier bimonoid in $\C$, equivalently,
$(t_2,t_1,e)$ is a multiplier bimonoid in $\C\rev$, if and only if the
assertions in parts (1) and (2) hold.
\end{proposition}

\begin{proof}
Let us again use the string notation 
\begin{equation*}
\begin{tikzpicture}[every node/.style={empty}]  
\path (0,1) node[name=s1t] {}
(0.4,1) node[name=s2t] {}
(0,0) node[name=s1b] {}
(0.4,0) node[name=s2b] {}
(0.2,0.5) node[arr,name=ti] {$t_i$}
(-0.5,0.5) node {$t_i=$}
(1.5,0.5) node{} ; 
\draw[braid] (s1t) to[out=270,in=135] (ti) ;
\draw[braid] (s2t) to[out=270,in=45] (ti) ;
\draw[braid] (s1b) to[out=90,in=225] (ti) ;
\draw[braid] (s2b) to[out=90,in=315] (ti) ;
\end{tikzpicture}~~~~
\begin{tikzpicture} 
\path (0,1) node[empty,name=s1t] {}
(1,1) node[empty,name=s2t]  {} 
(0.5,0.5) node[empty,name=m] {} 
(0.5,0) node[empty,name=sb] {}
(1.5,0.5) node{} ;
\path (-0.5,0.5) node {$m=$};
\draw[braid] (s1t) to[out=270,in=180] (m) to[out=0,in=270] (s2t);
\draw[braid] (m) to (sb) ;
\end{tikzpicture}~~~~
\begin{tikzpicture} 
\path (0,1) node[empty,name=s1t] {}
(0,0.3) node[empty,name=e] {$\circ$} 
(0.0,0) node[empty,name=sb] {}
(1.0,0.5) node{} ;
\path (-0.5,0.5) node {$e=$};
\draw[braid] (s1t) to (e);
\end{tikzpicture}~~~~
\begin{tikzpicture} 
\path (0,1) node[empty,name=s1t] {}
(1,1) node[empty,name=s2t] {}
(0,0) node[empty,name=s1b] {}
(1,0) node[empty,name=s2b] {}
(1.5,0.5) node{} ;
\path (-0.5,0.5) node {$b=$};
\draw[braid] (s2t) to (s1b);
\fill[white] (0.5,0.5) circle (0.1);
\draw[braid] (s1t) to (s2b);
\end{tikzpicture}
\end{equation*}

\noindent
for any $i=1,2,3,4$. We repeatedly use 
\begin{equation}\label{eq:a_12}
(m.1)(1.t_1)=(1.m)(t_2.1)
\end{equation}
which follows immediately from the axiom $(t_2.1)(1.t_1)=(1.t_1)(t_2.1)$ on
composing with $1.e.1$. 

(1) By Remark~\ref{rmk:short} it will suffice to prove the equivalence of the
short fusion equations. This follows via non-degeneracy from the following
calculations, which use \eqref{eq:a_12} and associativity.  

$$
\begin{tikzpicture} 
\path (0.8,1) node[arr,name=t1] {$t_1$} 
(0.5,1.6) node[empty,name=m1] {} 
(0.3,0.4) node[empty,name=m2] {};
  \path (0,2) node[empty,name=s1t] {}
(0.4,2) node[empty,name=s2t] {}
(0.6,2) node[empty,name=s3t] {}
(1,2) node[empty,name=s4t] {}
(0.3,0) node[empty,name=s1b] {}
(1,0) node[empty,name=s2b] {};
\draw[braid] (s2t) to[out=270,in=180] (m1) to[out=0,in=270] (s3t);
\draw[braid] (m1) to[out=270,in=135] (t1);
\draw[braid] (s4t) to[out=270,in=45] (t1);
\draw[braid] (t1) to[out=225,in=0] (m2) to[out=180,in=270] (s1t);
\draw[braid] (t1) to[out=315,in=90] (s2b);
\draw[braid] (m2) to (s1b);
\draw (1.6,1) node {$=$};
\path (2,1) node {};
\end{tikzpicture}
\begin{tikzpicture} 
\path (0.2,1) node[arr,name=t1] {$t_2$} 
(0.5,1.6) node[empty,name=m1] {} 
(0.8,0.4) node[empty,name=m2] {};
  \path (1,2) node[empty,name=s1t] {}
(0.6,2) node[empty,name=s2t] {}
(0.4,2) node[empty,name=s3t] {}
(0,2) node[empty,name=s4t] {}
(0.8,0) node[empty,name=s1b] {}
(0,0) node[empty,name=s2b] {};
\draw[braid] (s2t) to[out=270,in=0] (m1) to[out=180,in=270] (s3t);
\draw[braid] (m1) to[out=270,in=45] (t1);
\draw[braid] (s4t) to[out=270,in=135] (t1);
\draw[braid] (t1) to[out=315,in=180] (m2) to[out=0,in=270] (s1t);
\draw[braid] (t1) to[out=225,in=90] (s2b);
\draw[braid] (m2) to (s1b);
\end{tikzpicture}
$$

$$
\begin{tikzpicture}
  \path (1,2) node[arr,name=t1u] {$t_1$} (1,1) node[arr,name=t1d]  {$t_1$} (0,0.3) node[empty,name=m3] {} (-0.25,0) node[empty,name=m4] {};
\path (1.1,2.7) node[empty,name=l2s1t] {} (0.5,2.7) node[empty,name=l2s2t] {} (0,2.7) node[empty,name=l2s3t] {} (1,-0.5) node[empty,name=l2s1b] {} (-0.25,-0.5) node[empty,name=l2s2b] {} (-0.5,2.7) node[empty,name=l2s0t] {};
\draw[braid] (l2s1t) to[out=315,in=45] (t1u);
\draw[braid] (l2s2t) to[out=270,in=135] (t1u);
\draw[braid] (t1u) to[out=315,in=45] (t1d);
\draw[braid] (t1d) to[out=315,in=45] (l2s1b);
\path[braid,name path=s4] (l2s3t) to[out=270,in=135] (t1d);
\draw[braid,name path=s5] (t1u) to[out=225,in=0] (m3);
\fill[white, name intersections={of=s4 and s5}] (intersection-1) circle(0.1);
\path[braid,name path=s6] (t1d) to[out=210,in=180] (m3);
\fill[white, name intersections={of=s6 and s5}] (intersection-1) circle(0.1);
\draw[braid,name path=s4] (l2s3t) to[out=270,in=135] (t1d);
\draw[braid,name path=s6] (t1d) to[out=210,in=180] (m3);
\draw[braid] (m4) to (l2s2b);
\draw[braid] (l2s0t) to[out=270,in=180] (m4) to[out=0,in=270] (m3);
\draw (2,1.35) node {$=$};
\end{tikzpicture}
\begin{tikzpicture}
  \path (1,2) node[arr,name=t1u] {$t_1$} (1,1) node[arr,name=t1d]  {$t_1$} (0,0) node[empty,name=m3] {} (-0.25,0.3) node[empty,name=m4] {};
\path (1.1,2.7) node[empty,name=l2s1t] {} (0.5,2.7) node[empty,name=l2s2t] {} (0,2.7) node[empty,name=l2s3t] {} (1,-0.5) node[empty,name=l2s1b] {} (0,-0.5) node[empty,name=l2s2b] {} (-0.5,2.7) node[empty,name=l2s0t] {};
\draw[braid] (l2s1t) to[out=315,in=45] (t1u);
\draw[braid] (l2s2t) to[out=270,in=135] (t1u);
\draw[braid] (t1u) to[out=315,in=45] (t1d);
\draw[braid] (t1d) to[out=315,in=45] (l2s1b);
\path[braid,name path=s4] (l2s3t) to[out=270,in=135] (t1d);
\draw[braid,name path=s5] (t1u) to[out=225,in=0] (m3);
\fill[white, name intersections={of=s4 and s5}] (intersection-1) circle(0.1);
\path[braid,name path=s6] (t1d) to[out=210,in=0] (m4);
\fill[white, name intersections={of=s6 and s5}] (intersection-1) circle(0.1);
\draw[braid,name path=s4] (l2s3t) to[out=270,in=135] (t1d);
\draw[braid,name path=s6] (t1d) to[out=210,in=0] (m4);
\draw[braid] (m3) to (l2s2b);
\draw[braid] (l2s0t) to[out=270,in=180] (m4) to[out=270,in=180] (m3);
\draw (2,1.35) node {$=$};
\end{tikzpicture}
\begin{tikzpicture}
  \path (1,2) node[arr,name=t1u] {$t_1$} 
(0,2) node[arr,name=t2]  {$t_2$} 
(0,0) node[empty,name=m2] {} 
(1,0) node[empty,name=m1] {}
(0.5,1) node[empty,name=cr] {};
\path (1.1,2.7) node[empty,name=l2s1t] {} 
(0.9,2.7) node[empty,name=l2s2t] {} 
(0.1,2.7) node[empty,name=l2s3t] {} 
(-0.1,2.7) node[empty,name=l2s0t] {};
(1,-0.5) node[empty,name=l2s1b] {} 
(0,-0.5) node[empty,name=l2s2b] {} 
\draw[braid] (l2s1t) to[out=315,in=45] (t1u);
\draw[braid] (l2s2t) to[out=225,in=135] (t1u);
\draw[braid] (l2s3t) to[out=315,in=45] (t2);
\draw[braid] (l2s0t) to[out=225,in=135] (t2);
\draw[braid] (m1) to (l2s1b);
\draw[braid] (m2) to (l2s2b);
\path[braid,draw,name path=s1] (t2) to[out=225,in=180] (m2) to[out=0,in=180] (cr) to[out=0,in=225] (t1u);
\fill[white] (cr) circle(0.1);
\path[braid,draw,name path=s2] (t2) to[out=315,in=90] (cr) to[out=270,in=180] (m1) to[out=0,in=315] (t1u);
\draw (2,1.35) node {$=$};
\end{tikzpicture}
\begin{tikzpicture}
   \path (0,2) node[arr,name=t2u] {$t_2$} (0,1) node[arr,name=t2d]  {$t_2$} (1,0.3) node[empty,name=m] {}
(1.25,0) node[empty,name=m2] {};
\path (-0.1,2.7) node[empty,name=s1t] {} 
(0.1,2.7) node[empty,name=s2t] {} 
(1,2.7) node[empty,name=s3t] {} 
(1.5,2.7) node[empty,name=s4t] {} 
(0,-0.5) node[empty,name=s1b] {} 
(1.25,-0.5) node[empty,name=s2b] {} ;
\draw[braid] (s1t) to[out=225,in=135] (t2u);
\draw[braid] (s2t) to[out=315,in=45] (t2u);
\draw[braid] (t2u) to[out=225,in=135] (t2d);
\draw[braid] (t2d) to[out=225,in=135] (s1b);
\draw[braid,name path=s1] (s3t) to[out=270,in=45] (t2d);
\path[braid,name path=s2] (t2u) to[out=315,in=180] (m);
\fill[white, name intersections={of=s1 and s2}] (intersection-1) circle(0.1);
\draw[braid,name path=s3] (t2d) to[out=330,in=0] (m);
\fill[white, name intersections={of=s3 and s2}] (intersection-1) circle(0.1);
\draw[braid] (t2u) to[out=315,in=180] (m);
\draw[braid] (m2) to (s2b);
\draw[braid] (m) to[out=270,in=180] (m2) to[out=0,in=270] (s4t);
\end{tikzpicture}
$$

(2) Condition (iii) follows immediately from (i) by composing with the
counit. The reverse implication follows from~\eqref{eq:a_12} using
non-degeneracy. The equivalence of (ii) and (iii) follows by
duality.
\end{proof}

It follows by \eqref{eq:a_12} that, in a multiplier bimonoid $(t_1,t_2,e)$
with a non-degenerate multiplication, each of the morphisms $t_1$ and $t_2$
uniquely determines the other. 

\subsection{Regular multiplier bimonoid} \label{sec:reg_mbm}
Assuming some further structure, a more symmetric notion can be introduced.

\begin{definition} \label{def:reg_mbm}
A {\em regular multiplier bimonoid} in a braided monoidal category $\C$
consists of a multiplier bimonoid $(t_1,t_2)$ in $\C$ and a multiplier
bimonoid $(t_3,t_4)$ in $\overline\C$ with a common counit $e:A\to I$ such
that the following diagrams commute.
\begin{gather*} 
\xymatrix{
A^3\ar[r]^-{b.1}\ar[d]_-{1.t_1} \ar@{}[rrd]|-{\mathrm{(A)}} &
A^3\ar[r]^-{t_3.1}&
A^3\ar[d]^-{1.m} \\
A^3\ar[r]_-{b.1}&
A^3\ar[r]_-{1.m}&
A^2
}
\xymatrix{
A^3 \ar[r]^-{1.t_1} \ar[d]_{t_4.1} \ar@{}[rd]|-{\mathrm{(B)}} & A^3 \ar[d]^{t_4.1} \\
A^3 \ar[r]_-{1.t_1} & A^3} 
\\  
\xymatrix{
A^3\ar[r]^-{1.b}\ar[d]_-{t_2.1} \ar@{}[rrd]|-{\mathrm{(A\rev)}} & A^3\ar[r]^-{1.t_4}& A^3\ar[d]^-{m.1} \\
A^3\ar[r]_-{1.b} & A^3\ar[r]_-{m.1}& A^2}
\xymatrix{
A^3 \ar[r]^-{t_2.1} \ar[d]_{1.t_3} \ar@{}[rd]|-{\mathrm{(B\rev)}} & A^3 \ar[d]^{1.t_3} \\
A^3 \ar[r]_-{t_2.1} & A^3} 
\\ 
\xymatrix{
A^2\ar[r]^-{b^{-1}}\ar[d]_-{t_3} \ar@{}[rrd]|-{\mathrm{(C)}} & A^2\ar[r]^-{t_1}&
A^2\ar[d]^-{e.1}\\
A^2\ar[rr]_-{e.1}&& A}
\end{gather*}
\end{definition}

\begin{proposition}
  Given a multiplier bimonoid $(t_1,t_2,e)$ in $\C$ and a multiplier bimonoid
  $(t_3,t_4,e)$ in $\overline{\C}$, the following conditions are equivalent: 
  \begin{enumerate}
  \item $(t_1,t_2,t_3,t_4,e)$ is a regular multiplier bimonoid in $\C$;
  \item $(t_2,t_1,t_4,t_3,e)$ is a regular multiplier bimonoid in $\C\rev$;
  \item $(t_3,t_4,t_1,t_2,e)$ is a regular multiplier bimonoid in $\overline\C$;
  \item $(t_4,t_3,t_2,t_1,e)$ is a regular multiplier bimonoid in
    $\overline\C\rev$. 
  \end{enumerate}
\end{proposition}

\proof
We have seen that $(t_1,t_2,e)$ is a multiplier bimonoid in $\C$ just when
$(t_2,t_1,e)$ is a multiplier bimonoid in $\C\rev$; thus similarly
$(t_3,t_4,e)$ is a multiplier bimonoid in $\overline\C$ just when
$(t_4,t_3,e)$ is a multiplier bimonoid in $\overline\C\rev$.  Under this
duality, the axioms $\mathrm{(A)}$ and $\mathrm{(B)}$ correspond,
respectively, to the axioms $\mathrm{(A\rev)}$ and
$\mathrm{(B\rev)}$. Condition $\mathrm{(C)}$ says that the multiplication $m$
for the multiplier bimonoid $(t_1,t_2)$ is related to the multiplication
$\overline{m}$ for the multiplier bimonoid $(t_3,t_4)$ via the equation
$m=\overline{m}b$; this condition is self-dual. Thus (1) is equivalent to (2),
and (3) is equivalent to (4). In applying the $\C$-$\overline{\C}$ duality,
one must also replace $m$ by $\overline{m}=mb^{-1}$. Once again, this duality
interchanges $\mathrm{(B)}$ and $\mathrm{(B\rev)}$, and leaves $\mathrm{(C)}$
unchanged. Applied to $\mathrm{(A)}$ it gives an equivalent diagram (obtained
by composing both sides with various braid isomorphisms); the case of
$\mathrm{(A\rev)}$ is similar.  
\endproof

\begin{remark}\label{rmk:minimality}
The definition given above is in some sense a minimal one: we assume only
those axioms which will be needed to prove our results about modules and
comodules in the following sections. There are many further relationships
between the $t_i$ that follow from these in the non-degenerate case in which
we are primarily interested, and it may well be that some of these are needed
for the further development of the theory in the absence of non-degeneracy. In
particular, one might consider commutativity of diagrams such as the following
(or various dualizations). 
$$\xymatrix{
A^3 \ar[d]_{b.1} \ar[r]^-{1.t_1} & A^3 \ar[r]^-{b.1} & A^3 \ar[r]^-{1.t_1} &
A^3 \ar[d]^{t_3.1} \\ A^3 \ar[r]_{t_3.1} & A^3 \ar[rr]_-{1.t_1} && A^3 }~~~ 
\xymatrix{ 
A^3 \ar[r]^-{1.b} \ar[d]_{t_1.1} & A^3 \ar[r]^-{1.t_3} & A^3 \ar[r]^-{b.1} &
A^3 \ar[d]^-{1.t_1} \\  A^3 \ar[r]_-{1.b} & A^3 \ar[r]_-{1.t_3} & A^3
\ar[r]_{b.1} & A^3 } 
$$
\end{remark}

We now describe further simplifications which are possible in the
non-degenerate case.  

\begin{proposition}
Let $(t_1,t_2,e)$ define a multiplier bimonoid in the braided
monoidal category $\C$, and suppose that the corresponding multiplication
$m\colon A^2\to A$ is non-degenerate. Then morphism $t_3,t_4\colon A^2\to A^2$
define a regular multiplier bimonoid $(t_1,t_2,t_3,t_4,e)$ if
and only if the diagrams $\mathrm{(A)}$ and $\mathrm{(A\rev)}$ commute.  
\end{proposition}

\proof 
We need to prove the commutativity of $\mathrm{(B)}$, $\mathrm{(B\rev)}$, and
$\mathrm{(C)}$, as well as the fact that $(t_3,t_4,e)$ defines a multiplier
bimonoid in $\overline\C$; for the latter, we shall use
Proposition~\ref{prop:mb-nd}. 

Applying $e.1$ to either side of $\mathrm{(A)}$ and using non-degeneracy, we
deduce $\mathrm{(C)}$. Similarly, applying $1.e$ to either side of
$\mathrm{(A\rev)}$ and using non-degeneracy, we see that
$(1.e)t_4.b=(1.e)t_2=m$, and so that $(1.e)t_4=\overline{m}=(e.1)t_3$. Since
$\overline m$ is equal to $mb^{-1}$, it is non-degenerate; giving another one
of the hypotheses of Proposition~\ref{prop:mb-nd}. 

As for $\mathrm{(B)}$, we have 
\begin{equation*}
  \begin{tikzpicture}[every node/.style={empty}] 
    \path (0,2) node[name=s1t] {}
    (0.4,2) node[name=s2t] {}
    (0.6,2) node[name=s3t] {}
    (1.1,2) node[name=s4t] {}
    (0.0,0) node[name=s1b] {}
    (0.9,0) node[name=s2b] {}
    (1.1,0) node[name=s3b] {} 
    (0,0.75) node[name=m] {}
    (0.5,1.25) node[arr,name=t4] {$t_4$} 
    (1,0.75) node[arr,name=t1] {$t_1$} ;
    \draw[braid] (t1) to[out=315,in=45] (s3b);
    \draw[braid] (t1) to[out=225,in=135] (s2b);
    \draw[braid] (t4) to (t1);
    \draw[braid] (s4t) to[out=315,in=45] (t1);
    \draw[braid] (s3t) to[out=315,in=45] (t4);
    \draw[braid] (s2t) to[out=225,in=135] (t4);
    \draw[braid] (s1t) to[out=270,in=180] (m) to[out=0,in=270] (t4);
    \draw[braid] (m) to (s1b);
    \draw (1.8,1) node {$=$};
    \path (2.1,1) node{};
  \end{tikzpicture}
  \begin{tikzpicture}[every node/.style={empty}] 
    \path (0,2) node[name=s1t] {}
    (0.4,2) node[name=s2t] {}
    (0.6,2) node[name=s3t] {}
    (1.1,2) node[name=s4t] {}
    (0.2,0) node[name=s1b] {}
    (0.9,0) node[name=s2b] {}
    (1.1,0) node[name=s3b] {} 
    (0.6,1.25) node[name=v] {}
    (0.2,0.5) node[name=m] {}
    (0.2,1.25) node[arr,name=t2] {$t_2$} 
    (1,0.75) node[arr,name=t1] {$t_1$} ;
    \draw[braid] (t1) to[out=315,in=45] (s3b);
    \draw[braid] (t1) to[out=225,in=135] (s2b);
    \path[braid, name path=x] (t2) to (t1);
    \draw[braid] (s4t) to[out=315,in=45] (t1);
    \path[braid, name path=y] (s3t) to[out=315,in=45] (t2);
    \draw[braid, name path=z] (s2t) to[out=225,in=90] (v) to[out=270,in=0] (m) to[out=180,in=225] (t2);
\fill[white, name intersections={of=x and z}] (intersection-1) circle(0.1);
\fill[white, name intersections={of=y and z}] (intersection-1) circle(0.1);
    \draw[braid] (s1t) to[out=270,in=135] (t2); 
    \draw[braid] (t2) to (t1);
    \draw[braid] (s3t) to[out=315,in=45] (t2);
    \draw[braid] (m) to (s1b);
    \draw (1.8,1) node {$=$};
    \path (2.1,1) node{};
  \end{tikzpicture}
  \begin{tikzpicture}[every node/.style={empty}] 
    \path (0,2) node[name=s1t] {}
    (0.4,2) node[name=s2t] {}
    (0.9,2) node[name=s3t] {}
    (1.1,2) node[name=s4t] {}
    (0.2,0) node[name=s1b] {}
    (0.9,0) node[name=s2b] {}
    (1.1,0) node[name=s3b] {} 
    (0.6,1.25) node[name=v] {}
    (0.2,0.5) node[name=m] {}
    (0.2,1.0) node[arr,name=t2] {$t_2$} 
    (1,1.5) node[arr,name=t1] {$t_1$} ;
    \draw[braid] (t1) to[out=315,in=45] (s3b);
    \path[braid, name path=x] (t2) to[out=315,in=90]  (s2b);
    \draw[braid] (s4t) to[out=315,in=45] (t1);
    \path[braid, name path=y] (t1) to[out=225,in=45] (t2);
    \draw[braid, name path=z] (s2t) to[out=225,in=90] (v) to[out=270,in=0] (m) to[out=180,in=225] (t2);
\fill[white, name intersections={of=x and z}] (intersection-1) circle(0.1);
\fill[white, name intersections={of=y and z}] (intersection-1) circle(0.1);
    \draw[braid] (t1) to[out=225,in=45] (t2);
    \draw[braid] (s1t) to[out=270,in=135] (t2); 
    \draw[braid] (t2) to[out=315,in=90]  (s2b);
    \draw[braid] (s3t) to[out=225,in=135] (t1);
    \draw[braid] (m) to (s1b);
    \draw (1.8,1) node {$=$};
    \path (2.1,1) node{};
  \end{tikzpicture}
  \begin{tikzpicture}[every node/.style={empty}] 
    \path (0,2) node[name=s1t] {}
    (0.4,2) node[name=s2t] {}
    (0.9,2) node[name=s3t] {}
    (1.1,2) node[name=s4t] {}
    (0.0,0) node[name=s1b] {}
    (0.6,0) node[name=s2b] {}
    (1.1,0) node[name=s3b] {} 
    (0.6,1.25) node[name=v] {}
    (0.0,0.5) node[name=m] {}
    (0.5,1.0) node[arr,name=t4] {$t_4$} 
    (1,1.5) node[arr,name=t1] {$t_1$} ;
    \draw[braid] (t1) to[out=315,in=45] (s3b);
    \draw[braid] (s4t) to[out=315,in=45] (t1);
    \draw[braid] (s3t) to[out=225,in=135] (t1);
    \draw[braid] (s2t) to[out=225,in=135] (t4);
    \draw[braid] (t4) to[out=315,in=45] (s2b);
    \draw[braid] (t1) to (t4);
    \draw[braid] (s1t) to[out=250,in=180] (m) to[out=0,in=270] (t4);
    \draw[braid] (m) to (s1b);
   \end{tikzpicture}
\end{equation*}
where the first and last equalities use $\mathrm{(A\rev)}$ and the middle one
uses $(t_2.1)(1.t_1)=(1.t_1)(t_2.1)$; now $\mathrm{(B)}$ follows by
non-degeneracy, and dually $\mathrm{(B\rev)}$ also holds.

A similar (dual) argument applied to $\mathrm{(B)}$ shows that
$(t_4.1)(1.t_3)=(1.t_3)(t_4.1)$ holds; thus we are in a position to apply
Proposition~\ref{prop:mb-nd}. Now $e\overline{m}=emb^{-1}=(e.e)b^{-1}=e.e$,
and so $e$ is multiplicative with respect to the multiplication
$\overline{m}$. 

It remains to check that $t_3$ is a fusion morphism in
$\overline{\C}$; furthermore, by Remark~\ref{rmk:short}, it
suffices to check the short fusion equation. In the calculation
\begin{equation*}
\begin{tikzpicture}[every node/.style={empty}] 
  \path (1,2) node[arr,name=t1u] {$t_1$} 
(1,1) node[arr,name=t1d]  {$t_1$} 
(0,0.3) node[name=m1] {} 
(1,0) node[name=m2] {};
\path (1.1,2.7) node[name=s4t] {} 
(0.5,2.7) node[name=s3t] {} 
(0,2.7) node[name=s2t] {} 
(-0.5,2.7) node[name=s1t] {}
(1,-0.5) node[name=s2b] {} 
(0,-0.5) node[name=s1b] {} 
(-0.5,0.5) node[name=v] {}; 
\draw[braid] (s4t) to[out=315,in=45] (t1u);
\draw[braid] (s3t) to[out=270,in=135] (t1u);
\draw[braid] (t1u) to[out=315,in=45] (t1d);
\draw[braid] (t1d) to[out=315,in=0] (m2);
\draw[braid] (m2) to (s2b);
\draw[braid, name path=y] (m1) to (s1b);
\path[braid, name path=x] (s1t) to (v) to[out=270,in=180] (m2); 
\path[braid,name path=s4] (s2t) to[out=270,in=135] (t1d);
\draw[braid,name path=s5] (t1u) to[out=225,in=0] (m1);
\fill[white, name intersections={of=s4 and s5}] (intersection-1) circle(0.1);
\path[braid,name path=s6] (t1d) to[out=210,in=180] (m1);
\fill[white, name intersections={of=s6 and s5}] (intersection-1) circle(0.1);
\fill[white, name intersections={of=x and y}] (intersection-1) circle(0.1);
\draw[braid] (s1t) to (v) to[out=270,in=180] (m2); 
\draw[braid] (s2t) to[out=270,in=135] (t1d);
\draw[braid] (t1d) to[out=210,in=180] (m1);
\draw (2,1.35) node {$=$};
\end{tikzpicture}
\begin{tikzpicture}[every node/.style={empty}] 
  \path (1,2) node[arr,name=t1u] {$t_1$} 
(1,1) node[arr,name=t1d]  {$t_1$} 
(-0.5,0) node[name=m1] {} 
(1,0) node[name=m2] {};
\path (1.1,2.7) node[name=s4t] {} 
(0.5,2.7) node[name=s3t] {} 
(0,2.7) node[name=s2t] {} 
(-0.5,2.7) node[name=s1t] {}
(1,-0.5) node[name=s2b] {} 
(-0.5,-0.5) node[name=s1b] {} 
(-0.5,1.8) node[name=v1] {}
(0.7,0.5) node[name=v2] {}; 
\draw[braid] (s4t) to[out=315,in=45] (t1u);
\draw[braid] (s3t) to[out=270,in=135] (t1u);
\draw[braid] (t1u) to[out=315,in=45] (t1d);
\draw[braid] (t1d) to[out=315,in=0] (m2);
\draw[braid] (m2) to (s2b);
\draw[braid, name path=y] (m1) to (s1b);
\path[braid, name path=x] (s1t)to (v1) to[out=280,in=80] (v2) to[out=270,in=180] (m2); 
\path[braid,name path=s4] (s2t) to[out=270,in=135] (t1d);
\draw[braid,name path=s5] (t1u) to[out=225,in=0] (m1);
\fill[white, name intersections={of=s4 and s5}] (intersection-1) circle(0.1);
\path[braid,name path=s6] (t1d) to[out=210,in=180] (m1);
\fill[white, name intersections={of=s6 and s5}] (intersection-1) circle(0.1);
\fill[white, name intersections={of=x and s5}] (intersection-1) circle(0.1);
\draw[braid] (t1d) to[out=210,in=180] (m1);
\fill[white, name intersections={of=x and s6}] (intersection-1) circle(0.1);
\draw[braid] (s1t) to (v1) to[out=280,in=80]  (v2) to[out=240,in=180] (m2); 
\draw[braid] (s2t) to[out=270,in=135] (t1d);
\draw (2,1.35) node {$=$};
\end{tikzpicture}
\begin{tikzpicture}[every node/.style={empty}] 
  \path (1,2) node[arr,name=t1] {$t_1$} 
(-0.2,1.5) node[arr,name=t3]  {$t_3$} 
(-0.2,0) node[name=m1] {} 
(1,0.5) node[name=m2] {};
\path (1.1,2.7) node[name=s4t] {} 
(0.5,2.7) node[name=s3t] {} 
(0.1,2.7) node[name=s2t] {} 
(-0.5,2.7) node[name=s1t] {}
(1,-0.5) node[name=s2b] {} 
(-0.2,-0.5) node[name=s1b] {} 
(-0.5,1.8) node[name=v1] {}
(0.7,0.5) node[name=v2] {}; 
\draw[braid] (s4t) to[out=315,in=45] (t1);
\draw[braid] (s3t) to[out=270,in=135] (t1);
\path[braid, name path=v] (s1t) to[out=315,in=45] (t3);
\draw[braid, name path=u] (s2t) to[out=225,in=135] (t3);
\fill[white, name intersections={of=u and v}] (intersection-1) circle(0.1);
\draw[braid] (s1t) to[out=315,in=45] (t3);
\path[braid,name path=s4] (t3) to[out=315,in=180] (m2);
\draw[braid,name path=s5] (t1) to[out=225,in=0] (m1);
\fill[white, name intersections={of=s4 and s5}] (intersection-1) circle(0.1);
\draw[braid,name path=s4] (t3) to[out=315,in=180] (m2);
\draw[braid] (t1) to[out=315,in=0] (m2);
\draw[braid] (t3) to[out=225,in=180] (m1);
\draw[braid] (m2) to (s2b);
\draw[braid] (m1) to (s1b);
\draw (2,1.35) node {$=$};
\end{tikzpicture}
\begin{tikzpicture}[every node/.style={empty}] 
  \path (0.4,0.9) node[arr,name=t3d] {$t_3$} 
(-0.2,1.8) node[arr,name=t3]  {$t_3$} 
(-0.2,0) node[name=m1] {} 
(1,0.5) node[name=m2] {};
\path (1.1,2.7) node[name=s4t] {} 
(0.8,2.7) node[name=s3t] {} 
(0.1,2.7) node[name=s2t] {} 
(-0.5,2.7) node[name=s1t] {}
(1,-0.5) node[name=s2b] {} 
(-0.2,-0.5) node[name=s1b] {} 
(-0.5,1.8) node[name=v1] {}
(0.7,0.5) node[name=v2] {}; 
\path[braid, name path=v] (s1t) to[out=315,in=45] (t3);
\draw[braid, name path=u] (s2t) to[out=225,in=135] (t3);
\fill[white, name intersections={of=u and v}] (intersection-1) circle(0.1);
\draw[braid] (s1t) to[out=315,in=45] (t3);
\path[braid, name path=w] (t3) to[out=315,in=45] (t3d);
\draw[braid, name path=x] (s3t) to[out=245,in=135] (t3d);
\fill[white, name intersections={of=w and x}] (intersection-1) circle(0.1);
\draw[braid] (t3) to[out=315,in=45] (t3d);
\draw[braid] (t3) to[out=225,in=180] (m1) to[out=0,in=225] (t3d);
\draw[braid,name path=s4] (t3d) to[out=315,in=180] (m2) to[out=0,in=280] (s4t);
\draw[braid] (m2) to (s2b);
\draw[braid] (m1) to (s1b);
\end{tikzpicture}
\end{equation*}
the first equality holds by naturality of the braiding, and the second and
third by $\mathrm{(A)}$; while 
\begin{equation*}
\begin{tikzpicture}[every node/.style={empty}] 
  \path (1,1.2) node[arr,name=t1] {$t_1$} 
(0.5,1.6) node[name=m1] {} 
(1,0.5) node[name=m2] {};
\path (1.1,2.0) node[name=s4t] {} 
(0.7,2.0) node[name=s3t] {} 
(0.3,2.0) node[name=s2t] {} 
(0,2.0) node[name=s1t] {}
(1,0) node[name=s2b] {} 
(0,0) node[name=s1b] {};
\draw[braid] (s2t) to[out=270,in=180] (m1) to[out=0,in=270] (s3t);
\draw[braid] (m1) to[out=270,in=135] (t1);
\draw[braid] (s4t) to[out=300,in=45] (t1);
\path[braid,name path=u] (s1t) to[out=270,in=180] (m2);
\draw[braid,name path=v] (t1) to[out=225,in=90] (s1b);
\fill[white, name intersections={of=u and v}] (intersection-1) circle(0.1);
\draw[braid] (s1t) to[out=270,in=180] (m2) to[out=0,in=315] (t1);
\draw[braid] (m2) to (s2b);
\draw (2,1.35) node {$=$};
\end{tikzpicture}
\begin{tikzpicture}[every node/.style={empty}] 
  \path (0,0.8) node[arr,name=t3] {$t_3$} 
(0.5,1.6) node[name=m1] {} 
(1,0.3) node[name=m2] {};
\path (1.1,2.0) node[name=s4t] {} 
(0.7,2.0) node[name=s3t] {} 
(0.3,2.0) node[name=s2t] {} 
(0,2.0) node[name=s1t] {}
(1,0) node[name=s2b] {} 
(0,0) node[name=s1b] {};
\draw[braid] (s2t) to[out=270,in=180] (m1) to[out=0,in=270] (s3t);
\draw[braid,name path=u] (m1) to[out=270,in=135] (t3);
\draw[braid] (s4t) to[out=300,in=0] (m2) to[out=180,in=315] (t3);
\path[braid,name path=v] (s1t) to[out=270,in=45] (t3);
\draw[braid] (t3) to[out=225,in=90] (s1b);
\fill[white, name intersections={of=u and v}] (intersection-1) circle(0.1);
\draw[braid] (s1t) to[out=270,in=45] (t3);
\draw[braid] (m2) to (s2b);
\end{tikzpicture}
  \end{equation*}
holds by $\mathrm{(A)}$ once again. Since the left hand sides of the two
displayed calculations agree by the short fusion equation for
$t_1$, the right hand sides must also agree. By non-degeneracy we may cancel
the right-most input strings, and finally composing with suitably chosen braid
isomorphisms gives  
\begin{equation*}
  \begin{tikzpicture}[every node/.style={empty}] 
  \path (7,2) node[arr,name=t1u] {$t_3$} 
(7,1) node[arr,name=t1d]  {$t_3$} 
(6,0.3) node[empty,name=m3] {};
\path (7.1,2.7) node[empty,name=l2s1t] {} 
(6.5,2.7) node[empty,name=l2s2t] {} 
(6,2.7) node[empty,name=l2s3t] {} 
(7,0) node[empty,name=l2s1b] {} 
(6,0) node[empty,name=l2s2b] {};
\draw[braid] (l2s1t) to[out=315,in=45] (t1u);
\draw[braid] (l2s2t) to[out=270,in=135] (t1u);
\draw[braid] (t1u) to[out=315,in=45] (t1d);
\draw[braid] (t1d) to[out=315,in=45] (l2s1b);
\draw[braid,name path=s4] (l2s3t) to[out=270,in=135] (t1d);
\path[braid,name path=s5] (t1u) to[out=225,in=180] (m3);
\fill[white, name intersections={of=s4 and s5}] (intersection-1) circle(0.1);
\draw[braid] (t1u) to[out=225,in=180] (m3);
\draw[braid,name path=s6] (t1d) to[out=210,in=0] (m3);
\draw[braid] (m3) to (l2s2b);
\draw (8,1.35) node {$=$};
\path (9.6,1) node[arr,name=t1r] {$t_3$} 
(9.25,1.5) node[empty,name=m4] {};
\path (10.1,2.7) node[empty,name=r2s1t] {} 
(9.5,2.7) node[empty,name=r2s2t] {} 
(9,2.7) node[empty,name=r2s3t] {}
(10,0) node[empty,name=r2s1b] {} 
(9,0) node[empty,name=r2s2b] {};
\draw[braid] (r2s1t) to[out=250,in=45] (t1r);
\path[braid,name path=y] (r2s2t) to [out=225,in=180] (m4);
\draw[braid,name path=z] (r2s3t)  to [out=315,in=0] (m4);
\fill[white, name intersections={of=y and z}] (intersection-1) circle(0.1);
\draw[braid] (r2s2t) to [out=225,in=180] (m4);
\draw[braid] (m4) to[out=270,in=135] (t1r);
\draw[braid] (t1r) to[out=225,in=90] (r2s2b);
\draw[braid] (t1r) to[out=315,in=90] (r2s1b); 
  \end{tikzpicture}
\end{equation*}
which is the short fusion equation for $t_3$. \endproof

Just as for multiplier bialgebras over fields, for a regular
multiplier bimonoid $(t_1,t_2,t_3,t_4)$, any one of the maps
$t_1$, $t_2$, $t_3$, or $t_4$ determines each of the others whenever the
multiplication is non-degenerate; cf. axioms (A) and (A$\rev$) in Definition
\ref{def:reg_mbm}.

From \cite[Theorem~1.2]{Bohm:wmba_comod} we immediately obtain the following. 

\begin{example}\label{ex:reg_mbm_in_vec}
A multiplier bialgebra over a field is regular (in the sense of
\cite[Definition 1.1]{Bohm:wmba_comod}) if and only if the corresponding
multiplier bimonoid in $\mathsf{vec}$ in Example \ref{ex:mbm_from_mba} extends
to a regular multiplier bimonoid.
\end{example}

Another class of examples is provided by bimonoids in braided monoidal
categories:

\begin{example}
The multiplier bimonoid induced by a bimonoid $A$ in a braided monoidal
category $\C$ in Example \ref{ex:mbm_from_bimonoid} can be supplemented with
the morphisms 
$$
\xymatrix@R=10pt{
t_3:=\big(
A^2\ar[r]^-{d.1}&
A^3\ar[r]^-{1.b^{-1}}&
A^3\ar[r]^-{1.m}&
A^2\big)\\
t_4:=\big(
A^2\ar[r]^-{1.d}&
A^3\ar[r]^-{b^{-1}.1}&
A^3\ar[r]^-{m.1}&
A^2\big).}
$$
Very similar computations to those in Example \ref{ex:bimonoid} and Example
\ref{ex:mbm_from_bimonoid} show that it yields a regular multiplier bimonoid.
\end{example}

\section{Comodules and multiplier bicomonads}
\label{sec:comodule} 

As we have seen in Section \ref{sec:tricoc} in the case of counital fusion
morphisms, the best way to investigate the behavior of modules and 
comodules is to study the induced functors. In this section, therefore, we
generalize `bicomonads' (that is, monoidal comonads) to {\em multiplier
bicomonads} on arbitrary monoidal categories. We show that the monoidal
structure of the base category lifts to a suitably defined category of
comodules. Proving that any regular multiplier bimonoid induces a
multiplier bicomonad, we conclude that their comodules (in the appropriate
sense) constitute a monoidal category admitting a strict monoidal forgetful
functor to the base category.

\subsection{Multiplier bicomonad}

Based on the considerations in Section \ref{sec:tric_comod}, we start with the
following.

\begin{definition} \label{def:multi_bicomonad}
A {\em multiplier bicomonad} on a monoidal category $\C$ is a functor $G:\C\to
\C$ equipped with natural transformations $\check G_2:GX.GY\to G(GX.Y)$, $\hat
G_2:GX.GY\to G(X.GY)$ and $\varepsilon:GX\to X$ such that $(\check
G_2,\varepsilon)$ makes $G$ a right multiplier bicomonad on $\C$,
$(\hat G_2,\varepsilon)$ makes $G$ a left multiplier bicomonad on
$\C$, and the following diagrams, expressing their compatibility, commute for
any objects $X,Y,Z$. First,
$$
\xymatrix{
GX.GY\ar[r]^-{\check G_2}\ar[d]_-{\hat G_2}&
G(GX.Y)\ar[d]^-{G(\varepsilon.1)}\\
G(X.GY)\ar[r]_-{G(1.\varepsilon)}&
G(X.Y).}
$$
The common diagonal in this diagram satisfies the same associativity condition
as the binary part of a monoidal functor (see Section
\ref{sec:tric_comod}). For this reason --- although in general it does not 
admit for a nullary part --- we use the notation $G_2:GX.GY\to G(X.Y)$ for
it. We also require this $G_2$ to satisfy the second compatibility
condition
$$
\xymatrix{
GX.GY.GZ\ar[r]^-{1.\check G_2}\ar[d]_-{\hat G_2.1}&
GX.G(GY.Z)\ar[d]^-{G_2}\\
G(X.GY).GZ\ar[r]_-{G_2}&
G(X.GY.Z).}
$$
\end{definition}

\begin{example}\label{ex:multi_comonad_from_bicomonad}
Consider a monoidal comonad $(G,\delta,\varepsilon)$ on a monoidal category
$\C$. We know from Example \ref{ex:G_check_from_bicomonad} that
$$
\check G_2:=\big(
\xymatrix{
GX.GY\ar[r]^-{\delta.1}&
G^2X.GY\ar[r]^-{G_2}&
G(GX.Y)}\big)
$$
(together with $\varepsilon$) makes $G$ into a right multiplier
bicomonad, and similarly
$$
\hat G_2:=\big(
\xymatrix{
GX.GY\ar[r]^-{1.\delta}&
GX.G^2Y\ar[r]^-{G_2}&
G(X.GY)}\big)
$$
makes $G$ into a left multiplier bicomonad. We claim that they constitute a
multiplier bicomonad. Indeed,
by the naturality of $G_2$ and by the counitality of $\delta$, also
$$
\xymatrix{
GX.GY\ar[r]^-{1.\delta}\ar[d]_-{\delta.1}\ar@{=}[rd]&
GX.G^2Y\ar[r]^-{G_2}\ar[d]^-{1.G\varepsilon}&
G(X.GY)\ar[dd]^-{G(1.\varepsilon)}\\
G^2X.GY\ar[r]^-{G\varepsilon.1}\ar[d]_-{G_2}&
GX.GY\ar[rd]^-{G_2}\\
G(GX.Y)\ar[rr]_-{G(\varepsilon.1)}&&
G(X.Y)}
$$
commutes and so does the second compatibility diagram in Definition
\ref{def:multi_bicomonad} by the associativity of $G_2$.
\end{example}

\begin{example}\label{ex:multi_comonad_from_t}
Consider a regular multiplier bimonoid $(t_1,t_2,t_3,t_4:A^2\to A^2,e:A\to I)$
in a braided monoidal category $\C$; and the induced functor $G=(-).A:\C\to
\C$. We know from Example \ref{ex:G_check_from_t} that $\varepsilon:=1.e$ and   
$$
\check G_2:=\big(
\xymatrix{
X.A.Y.A\ar[r]^-{1.b.1}&
X.Y.A^2\ar[r]^-{1.1.t_1}&
X.Y.A^2\ar[r]^-{1.b^{-1}.1}&
X.A.Y.A}\big)
$$
make $G$ into a right multiplier bicomonad. We claim that together with 
$$
\hat G_2:=\big(
\xymatrix{
X.A.Y.A\ar[r]^-{1.b.1}&
X.Y.A^2\ar[r]^-{1.1.b}&
X.Y.A^2\ar[r]^-{1.1.t_3}&
X.Y.A^2}\big)
$$
they constitute a multiplier bicomonad. Certainly $\hat G_2$ provides a 
left multiplier bicomonad structure; the first compatibility condition in
Definition \ref{def:multi_bicomonad} follows from axiom (C)  in
Definition \ref{def:reg_mbm} (together with the naturality and the coherence
of the braiding) and the second one follows by axiom (A) in
Definition \ref{def:reg_mbm}.
\end{example}

\subsection{The category of comodules}

Based on the notion of comodule in Definition \ref{def:G_check_comodule}, we
introduce the following. 

\begin{definition}\label{def:G_comodule}
By a {\em comodule} over a multiplier bicomonad $G$ on a monoidal category
$\C$ we mean the following. It is an object $V$ in $\C$ equipped with the
structure $\check v:V.G(-)\to G(V.-)$ of a comodule over the right 
multiplier bicomonad $(G,\check G_{2},\varepsilon)$, and also with the
structure $\hat v:G(-).V\to G(-.V)$ of a comodule over the left 
multiplier bicomonad $(G,\hat G_{2},\varepsilon)$ such that the
compatibility diagram 
$$
\xymatrix{
GX.V.GY\ar[r]^-{1.\check v}\ar[d]_-{\hat v.1}&
GX.G(V.Y)\ar[d]^-{G_2}\\
G(X.V).GY\ar[r]_-{G_2}&
G(X.V.Y)}
$$
commutes for any objects $X,Y$ (where $G_2$ is the natural transformation
introduced in Definition \ref{def:multi_bicomonad}).
A {\em morphism} of comodules is a morphism $f:V\to W$ in $\C$ such that both
diagrams
$$
\xymatrix{
V.GX\ar[d]_-{f.1}\ar[r]^-{\check v}&
G(V.X)\ar[d]^-{G(f.1)} &
GX.V\ar[d]_-{1.f}\ar[r]^-{\hat v}&
G(X.V)\ar[d]^-{G(1.f)}\\
W.GX\ar[r]_-{\check w}&
G(W.X) &
GX.W\ar[r]_-{\hat w}&
G(X.W)}
$$
commute for any object $X$. 
\end{definition}

\begin{theorem}\label{thm:mon_cat_of_G_comodules}
Consider a multiplier bicomonad $G$ on a monoidal category $\C$. The monoidal
structure of $\C$ lifts to the category of $G$-comodules in Definition
\ref{def:G_comodule}. 
\end{theorem}

\begin{proof}
By Theorem \ref{thm:mon_cat_of_check_G_comodules}, the monoidal unit $I$
carries coassociative and counital comodule structures
$$
\xymatrix{
I.G(-)\ar[r]^-\cong&
G\ar[r]^-\cong&
G(I.-)&
G(-).I\ar[r]^-\cong&
G\ar[r]^-\cong&
G(-.I),}
$$
built up from the unit isomorphisms in $\C$.
The diagram in Definition \ref{def:G_comodule} commutes for this $G$-comodule
$I$ by naturality of $G_2$ and the coherence in $\C$. By Theorem
\ref{thm:mon_cat_of_check_G_comodules}, for any two $G$-comodules $V$ and $W$, 
there are coassociative and counital comodule structures
$$
\xymatrix @R=10pt{
V.W.G(-)\ar[r]^-{1.\check w}&
V.G(W.-)\ar[r]^-{\check v}&
G(V.W.-)\\
G(-).V.W\ar[r]^-{\hat v.1}&
G(-.V).W\ar[r]^-{\hat w}&
G(-.V.W).}
$$
Since both $V$ and $W$ satisfy the compatibility condition in Definition
\ref{def:G_comodule}, it follows by the functoriality of the monoidal product
that also
$$
\xymatrix{
GX.V.W.GY\ar[r]^-{1.1.\check w}\ar[d]_-{\hat v.1.1}&
GX.V.G(W.Y)\ar[r]^-{1.\check v}\ar[d]^-{\hat v.1}&
GX.G(V.W.Y)\ar[dd]^-{G_2}\\
G(X.V).W.GY\ar[r]^-{1.\check w}\ar[d]_-{\hat w.1}&
G(X.V).G(W.Y)\ar[rd]^-{G_2}\\
G(X.V.W).GY\ar[rr]_-{G_2}&&
G(X.V.W.Y)}
$$
commutes for any objects $X,Y$. In light of Theorem
\ref{thm:mon_cat_of_check_G_comodules} this completes the proof.
\end{proof}

\begin{example}
Consider the multiplier bicomonad on a monoidal category $\C$ induced by a
bicomonad $(G,\delta,\varepsilon)$ as in Example
\ref{ex:multi_comonad_from_bicomonad}. We claim that the category of comodules
over it, in the sense of Definition \ref{def:G_comodule}, is isomorphic to the
Eilenberg-Moore category of comodules over the comonad $G$; hence Theorem
\ref{thm:mon_cat_of_G_comodules} extends the result about the lifting of the
monoidal structure of $\C$ to the Eilenberg-Moore category of $G$ (see
e.g. \cite{McCrudden}).

The stated isomorphism acts on the morphisms as the identity map and its
object map is the following. Let us begin with an Eilenberg-Moore comodule
$v:V\to GV$ and put
$$
\check v:=\big(
\xymatrix{
V.GX\ar[r]^-{v.1}&
GV.GX\ar[r]^-{G_2}&
G(V.X) \big) }\ \ 
\hat v:=\big(
\xymatrix{
GX.V\ar[r]^-{1.v}&
GX.GV\ar[r]^-{G_2}&
G(X.V) \big) .}
$$
They satisfy the conditions in \eqref{eq:v_check} and their opposites,
respectively, see Example \ref{ex:EM_comodule}, and they obey the
compatibility condition in Definition \ref{def:G_comodule} by the
associativity of $G_2$. 
Conversely, let $(V,\check v,\hat v)$ be a $G$-comodule as in Definition
\ref{def:G_comodule}. We know from Example \ref{ex:EM_comodule} that the
natural transformation $\check v$ corresponds bijectively to the
Eilenberg-Moore coaction 
$$
\xymatrix{
V\ar[r]^-{1.G_0}&
V.GI\ar[r]^-{\check v}&
GV,}
$$
and the natural transformation $\hat v$ corresponds bijectively to the
Eilenberg-Moore coaction 
$$
\xymatrix{
V\ar[r]^-{G_0.1}&
GI.V\ar[r]^-{\hat v}&
GV.}
$$
Precomposing by 
$
\xymatrix@C=35pt{
V \ar[r]^-{G_0.1.G_0}&
GI.V.GI}
$
both paths around the compatibility diagram in Definition \ref{def:G_comodule}
at $X=Y=I$, and using the unitality of the monoidal structure of $G$, we
conclude that these Eilenberg-Moore coactions coincide.
\end{example}

Further examples of the situation in Definition \ref{def:G_comodule} are
provided by the following.

\begin{definition}\label{def:reg_mbm_comodule}
A {\em comodule} over a regular multiplier bimonoid $(t_1,t_2,t_3,t_4:A^2\to
A^2,e:A\to I)$ in a braided monoidal category $\C$ is an object
$V$ of $\C$ equipped with the structure $v_1:V.A\to V.A$ of a comodule over
the counital fusion morphism $t_1$ in $\C$ and the structure $v_3:V.A\to V.A$
of a comodule over the counital fusion morphism $t_3$ in $\overline \C$ such
that the following diagram commutes.
$$
\xymatrix{
A.V.A\ar[r]^-{1.v_1}\ar[d]_-{b.1}&
A.V.A\ar[r]^-{b.1}&
V.A^2\ar[d]^-{1.m}\\
V.A^2\ar[r]_-{v_3.1}&
V.A^2\ar[r]_-{1.m}&
V.A}
$$
where $m$ is the multiplication in Proposition
\ref{prop:multiplication}~(1) (associated to $t_1$ or $t_2$, cf. the second
diagram in Definition \ref{def:mbm}; differing by the braiding from the
multiplication associated to $t_3$ or $t_4$; cf. axiom (C)  in
Definition \ref{def:reg_mbm}). 
A {\em morphism} of comodules is a morphism $f:V\to W$ in $\C$ such that the
following diagrams commute.
$$
\xymatrix{
V.A\ar[d]_-{f.1}\ar[r]^-{v_1}&
V.A\ar[d]^-{f.1} &
V.A\ar[d]_-{f.1}\ar[r]^-{v_3}&
V.A\ar[d]^-{f.1}\\
W.A\ar[r]_-{w_1}&
W.A &
W.A\ar[r]_-{w_3}&
W.A}
$$
\end{definition}

It follows immediately from the compatibility condition in Definition
\ref{def:reg_mbm_comodule} that, in a comodule $(V,v_1,v_3)$ over a regular
multiplier bimonoid with {\em non-degenerate} multiplication, $v_3$ is
uniquely determined by $v_1$. Equivalently, $v_1$ is uniquely determined by
$v_3$. Furthermore, in this case, both diagrams in Definition
\ref{def:reg_mbm_comodule} defining morphisms of comodules become equivalent
to each other: morphisms of comodules can be defined by either one of
them. 

For a regular multiplier bimonoid in $\mathsf{vec}$, induced by a
regular multiplier bialgebra over a field as in Example
\ref{ex:reg_mbm_in_vec}, we recover the notion of comodule in
\cite[Definition 2.7]{VDaZha:corep_I}. 

A comodule $(V,v_1,v_3)$ as in Definition \ref{def:reg_mbm_comodule} induces a
comodule over the multiplier bicomonad in Example
\ref{ex:multi_comonad_from_t} by putting
$$
\xymatrix@R=10pt{
\check v:=\big(
V.X.A\ar[r]^-{b.1}&
X.V.A\ar[r]^-{1.v_1}&
X.V.A\ar[r]^-{b^{-1}.1}&
V.X.A\big)\\
\hat v:=\big(
X.A.V\ar[r]^-{1.b}&
X.V.A\ar[r]^-{1.v_3}&
X.V.A\big).}
$$
Hence from Theorem \ref{thm:mon_cat_of_G_comodules} we obtain the following.

\begin{corollary}
For any regular multiplier bimonoid $(t_1,t_2,t_3,t_4:A^2\to A^2,e:A\to I)$ in
a braided monoidal category $\C$, the monoidal structure of $\C$ lifts to the
category of comodules in Definition \ref{def:reg_mbm_comodule}.
\end{corollary}

For any comodules $(V,v_1,v_3)$ and $(W,w_1,w_3)$, $V.W$ is
again a comodule via
$$
\xymatrix@R=10pt{
V.W.A\ar[r]^-{1.w_1}&
V.W.A\ar[r]^-{b.1}&
W.V.A\ar[r]^-{1.v_1}&
W.V.A\ar[r]^-{b^{-1}.1}&
V.W.A\\
V.W.A\ar[r]^-{1.b^{-1}}&
V.A.W\ar[r]^-{v_3.1}&
V.A.W\ar[r]^-{1.b}&
V.W.A\ar[r]^-{1.w_3}&
V.W.A.}
$$

\section{Modules and multiplier bimonads}
\label{sec:module}

This section is devoted to the study of the modules over a regular multiplier
bimonoid in a braided monoidal category. As in the case of comodules in the
previous section, this will be done by investigating the induced functors. To
this end, we define {\em multiplier bimonads} on arbitrary monoidal
categories, which generalize `bimonads'; that is, opmonoidal monads. Under
some further assumptions (of certain morphisms being split epimorphisms) we
show that the monoidal structure of the base category lifts to the category of
suitably defined modules. Showing that any regular multiplier bimonoid induces
a multiplier bimonad, we draw conclusions about the categories of their
modules.

\subsection{Multiplier bimonad}
Based on the considerations in Section \ref{sec:tric_module}, we introduce the
following notion.

\begin{definition} \label{def:multi_bimonad}
A {\em multiplier bimonad} on a monoidal category $\C$ is a functor $T:\C\to
\C$ equipped with natural transformations $\hat T_2:T(X.TY)\to TX.TY$, $\check
T_2:T(TX.Y)\to TX.TY$ and a morphism $T_0:TI\to I$ such that $(\hat T_2,T_0)$
makes $T$ a left multiplier bimonad on $\C$, $(\check T_2,T_0)$
makes $T$ a right multiplier bimonad on $\C$, and these structures
are compatible in the sense of the following commutative diagrams.
$$
\xymatrix{
T(TX.Y.TZ)\ar[r]^-{\check T_2}\ar[d]_-{\hat T_2}&
TX.T(Y.TZ)\ar[d]^-{1.\hat T_2} &
T^2X\ar[r]^-{\hat T_2}\ar[d]_-{\check T_2}&
TI.TX\ar[d]^-{T_0.1}\\
T(TX.Y).TZ\ar[r]_-{\check T_2.1}&
TX.TY.TZ &
TX.TI\ar[r]_-{1.T_0}&
 TX}
$$
\end{definition}

The common diagonal in the second diagram in Definition
\ref{def:multi_bimonad} is an associative (though in general non-unital)
multiplication that we denote by $\mu:T^2\to T$.

By the compatibility diagrams in Definition \ref{def:multi_bimonad}, the
following diagram commutes for any multiplier bimonad $T$ and any objects
$X,Y$. 
\begin{equation}\label{eq:same_surjective}
\xymatrix{
T(TX.TY)\ar[r]^-{\hat T_2}\ar[d]_-{\check T_2}&
T^2X.TY\ar[r]^-{\hat T_2.1}\ar[d]^-{\check T_2.1}&
TI.TX.TY\ar[dd]^-{T_0.1.1}\\
TX.T^2Y\ar[r]^-{1.\hat T_2}\ar[d]_-{1.\check T_2}&
TX.TI.TY\ar[rd]^-{1.T_0.1}\\
TX.TY.TI\ar[rr]_-{1.1.T_0}&&
TX.TY}
\end{equation}

\begin{example}\label{ex:multi_bimonad_from_bimonad}
Consider an opmonoidal monad $(T,\mu,\eta)$ on a monoidal category $\C$. We
know from Example \ref{ex:T_hat_from_bimonad} that 
$$
\hat T_2:=\big(
\xymatrix{
T(X.TY)\ar[r]^-{T_2}&
TX.T^2Y\ar[r]^-{1.\mu}&
TX.TY}\big)
$$
(together with the nullary part $T_0:TI\to I$ of the opmonoidal
structure) makes $T$ into a left multiplier bimonad, and similarly
$$
\check T_2:=\big(
\xymatrix{
T(TX.Y)\ar[r]^-{T_2}&
T^2X.TY\ar[r]^-{\mu.1}&
TX.TY}\big)
$$
and $T_0$ make $T$ into a right multiplier bimonad. They obey the
compatibility conditions in Definition \ref{def:multi_bimonad} by the
coassociativity and the counitality of the opmonoidal structure $(T_2,T_0)$
and the functoriality of the monoidal product. 
\end{example}

\begin{example}\label{ex:multi_bimonad_from_t}
Consider a regular multiplier bimonoid $(t_1,t_2,t_3,t_4:A^2\to A^2,e:A\to I)$
in a braided monoidal category $\C$ and the induced functor $T=A.(-):\C\to
\C$. We know from Example \ref{ex:T_hat_from_t} that $T_0=e$ and 
$$
\hat T_2:=\big(
\xymatrix{
A.X.A.Y\ar[r]^-{1.b^{-1}.1}&
A^2.X.Y\ar[r]^-{t_1.1.1}&
A^2.X.Y\ar[r]^-{1.b.1}&
A.X.A.Y}\big)
$$
make $T$ into a left multiplier bimonad. We claim that together with
$$
\check T_2:=\big(
\xymatrix{
A^2.X.Y\ar[r]^-{b.1.1}&
A^2.X.Y\ar[r]^-{t_4.1.1}&
A^2.X.Y\ar[r]^-{1.b.1}&
A.X.A.Y}\big)
$$
they constitute a multiplier bimonad. Certainly $\check T_2$ provides a 
right multiplier bimonad structure; the compatibility conditions in 
Definition~\ref{def:multi_bimonad} hold by axiom (B), and the equivalent form 
$(1.e)t_4b=(e.1)t_1$ of axiom (C) in Definition \ref{def:reg_mbm},
respectively.
\end{example}

\subsection{The category of modules} Based on the notion of module in
Definition \ref{def:T_hat_module}, we introduce the following.

\begin{definition}\label{def:multi_bimonad_module}
A {\em module} over a multiplier bimonad on a monoidal category $\C$ is an
object $Q$ in $\C$ equipped with the structure $\hat q:T(-.Q)\to
T(-).Q$ of a module over the left multiplier bimonad $(T,\hat
T_2,T_0)$, and also with the structure $\check q:T(Q.-)\to Q.T(-)$ of a module
over the right multiplier bimonad $(T,\check T_2,T_0)$ and these
structures are compatible in the sense of the following commutative diagrams
(for any objects 
$X,Y$).
$$
\xymatrix{
T(TX.Y.Q)\ar[r]^-{\hat q}\ar[d]_-{\check T_2}&
T(TX.Y).Q\ar[d]^-{\check T_2.1} &
T(Q.X.TY)\ar[r]^-{\check q}\ar[d]_-{\hat T_2}&
Q.T(X.TY)\ar[d]^-{1.\hat T_2} \\
TX.T(Y.Q)\ar[r]_-{1.\hat q}&
TX.TY.Q&
T(Q.X).TY\ar[r]_-{\check q.1}&
Q.TX.TY \\
&TQ \ar[r]^-{\hat q}\ar[d]_-{\check q}&
TI.Q\ar[d]^-{T_0.1}\\
&Q.TI\ar[r]_-{1.T_0}&
Q.}
$$
A {\em morphism} of modules is a morphism $f:Q\to P$ in $\C$ rendering
commutative the diagrams
$$
\xymatrix{
T(X.Q)\ar[r]^-{\hat q}\ar[d]_-{T(1.f)}&
TX.Q\ar[d]^-{1.f} &
T(Q.X)\ar[r]^-{\check q} \ar[d]_-{T(f.1)}&
Q.TX\ar[d]^-{f.1}\\
T(X.P)\ar[r]_-{\hat p}&
TX.P &
T(P.X)\ar[r]_-{\check p}&
P.TX.}
$$
for any object $X$.
\end{definition}

The common diagonal in the third diagram in Definition
\ref{def:multi_bimonad_module} is an associative action (with respect to the
associative multiplication $\mu:T^2\to T$) that we denote by $q:TQ\to Q$.
It is a split epimorphism by Definition \ref{def:T_hat_module}. 

\begin{theorem}\label{thm:mon_cat_of_multi_bimonad_modules}
Consider a multiplier bimonad $T$ on a monoidal category $\C$. Assume that
$T_0$ and, for any objects $X$ and $Y$, the equal morphisms
(cf. \eqref{eq:same_surjective}) 
$$
\xymatrix@ R=8pt{
\big(T(TX.TY)\ar[r]^-{\hat T_2}&
T^2X.TY\ar[r]^-{\hat T_2.1}&
TI.TX.TY\ar[r]^-{T_0.1.1}&
TX.TY\big)=\\
\big(T(TX.TY)\ar[r]^-{\check T_2}&
TX.T^2Y\ar[r]^-{1.\check T_2}&
TX.TY.TI\ar[r]^-{1.1.T_0}&
TX.TY\big)}
$$
are split epimorphisms. Then the category of $T$-modules of Definition
\ref{def:multi_bimonad_module} is monoidal in such a way that the forgetful
functor to $\C$ is strict monoidal.
\end{theorem}

\begin{proof}
By Theorem \ref{thm:mon_cat_of_hat_T_modules} and by the coherence in $\C$, the
monoidal unit $I$ of $\C$ carries a $T$-module structure (via the natural
transformations built up from the unit isomorphisms). Also by
Theorem \ref{thm:mon_cat_of_hat_T_modules}, the monoidal product of any
$T$-modules $P$ and $Q$ is a $T$-module via the natural transformations
$$
\xymatrix@R=8pt{
T(-.P.Q)\ar[r]^-{\hat q}&
T(-.P).Q\ar[r]^-{\hat p.1}&
TX.P.Q\\
T(P.Q.X)\ar[r]^-{\check p}&
P.T(Q.X)\ar[r]^-{1.\check q}&
P.Q.TX.}
$$
These natural transformations clearly obey the first two compatibility
conditions in Definition \ref{def:multi_bimonad_module} whenever $(\hat
q,\check q)$ and $(\hat p,\check p)$ do. Since $q=(T_0.1)\hat q$ is
a split epimorphism, the equal morphisms in the top row and in the
left column of the diagram
$$ 
\xymatrix@C=27pt{
T(P.TQ)\ar@{=}[r]\ar[dddd]_-{T(1.\hat q)}&
T(P.TQ)\ar@{=}[r]\ar[ddd]_-{\check p}&
T(P.TQ)\ar[r]^-{T(1.\hat q)}\ar[dd]^-{\hat T_2}\ar@{}[rdd]|-{\textrm{(d)}}&
T(P.TI.Q)\ar[r]^-{T(1.T_0.1)}\ar[d]^-{\hat q}\ar@{}[rd]|-{\textrm{(b)}}&
T(P.Q)\ar[d]^-{\hat q}\\
&&&
T(P.TI).Q\ar[r]^-{T(1.T_0).1}\ar[d]^-{\hat T_2.1}\ar@{}[rd]|-{\textrm{(c)}}&
TP.Q\ar@{=}[d]\\
\ar@{}[r]|-{\textrm{(a)}}&
\ar@{}[r]|-{\textrm{(e)}}&
TP.TQ\ar[r]^-{1.\hat q}\ar[d]^-{\check p.1}&
TP.TI.Q \ar[r]^-{1.T_0.1}\ar[d]^-{\check p.1.1}&
TP.Q\ar[ddd]^-{\check p.1}\\
&
P.T^2Q\ar[r]^-{1.\hat T_2}\ar[d]_-{1.T\hat q}\ar@{}[rrd]|-{\textrm{(d)}}&
P.TI.TQ\ar[r]^-{1.1.\hat q}&
P.TI.TI.Q\ar@{=}[d]\\
T(P.TI.Q)\ar[d]_-{T(1.T_0.1)}&
P.T(TI.Q)\ar[r]^-{1.\hat q}\ar[d]_-{1.T(T_0.1)}\ar@{}[rd]|-{\textrm{(b)}}&
P.T^2I.Q\ar[r]^-{1.\hat T_2.1}\ar[d]^-{1.TT_0.1}\ar@{}[rd]|-{\textrm{(c)}}&
P.TI.TI.Q\ar[rd]^-{1.1.T_0.1}\\
T(P.Q)\ar[r]_-{\check p}&
P.TQ\ar[r]_-{1.\hat q}&
P.TI.Q\ar@{=}[rr]&&
P.TI.Q}
$$
are epimorphisms. The regions labelled by (a) and (b) commute by the
naturality of $\check p$ and of $\hat q$, respectively. Regions (c) commute by
the counitality of $\hat T_2$. Regions (d) commute by the fusion equation
on $\hat q$ and (e) commutes since the second compatibility condition in
Definition \ref{def:multi_bimonad_module} holds on $P$. The unlabelled regions
commute by the functoriality of the monoidal product.
This proves that the morphisms in the right column and the bottom row are
equal; that is, $(1.\hat q)\check p=(\check p.1)\hat q$. Using this identity
and the assumption that $P$ and $Q$ obey the third compatibility condition in
Definition \ref{def:multi_bimonad_module}, we conclude that also $P.Q$ obeys
the third compatibility condition in Definition
\ref{def:multi_bimonad_module}:
$$
\xymatrix{
T(P.Q)\ar[r]^-{\hat q}\ar[d]_-{\check p}&
TP.Q\ar[r]^-{\hat p.1}\ar[d]^-{\check p.1}&
TI.P.Q\ar[dd]^-{T_0.1.1}\\
P.TQ\ar[r]^-{1.\hat q}\ar[d]_-{1.\check q}&
P.TI.Q\ar[rd]^-{1.T_0.1}\\
P.Q.TI\ar[rr]_-{1.1.T_0}&&
P.Q\ .}
$$
In light of Theorem \ref{thm:mon_cat_of_hat_T_modules}, this
completes the proof.
\end{proof}

\begin{example}
Consider an opmonoidal monad $(T,\mu,\eta)$ on a monoidal category $\C$ and
the induced multiplier bimonad in Example
\ref{ex:multi_bimonad_from_bimonad}. We claim that its category of modules in 
the sense of Definition \ref{def:multi_bimonad_module} is isomorphic to the
usual Eilenberg-Moore category of modules. Hence Theorem
\ref{thm:mon_cat_of_multi_bimonad_modules} generalizes the fact (see
e.g. \cite{McCrudden}) that the monoidal structure of the base category lifts
to the Eilenberg-Moore category of a bimonad.

The stated isomorphism acts on the morphisms as the identity map. It takes an
Eilenberg-Moore module $q:TQ\to Q$ to the module
$$
\xymatrix{
T(-.Q)\ar[r]^-{T_2}&
T(-).TQ\ar[r]^-{1.q}&
T(-).Q &
T(Q.-)\ar[r]^-{T_2}&
TQ.T(-)\ar[r]^-{q.1}&
Q.T(-)}
$$
in the sense of Definition \ref{def:multi_bimonad_module}.
They obey the conditions in Definition \ref{def:T_hat_module} and their
opposites, respectively, see Example \ref{ex:EM_module}. The compatibility
conditions in Definition \ref{def:multi_bimonad_module} hold by the
coassociativity and counitality of the opmonoidal structure $(T_2,T_0)$ and
the functoriality of the monoidal product.
In the opposite direction, a module $(\hat q:T(-.Q)\to T(-).Q,\check q:T(Q.-)
\to Q.T(-))$ in the sense of Definition \ref{def:multi_bimonad_module} is
taken to the Eilenberg-Moore comodule $q:=(T_0.1)\hat q=(1.T_0).\check q$. We
know from Example \ref{ex:EM_module} that this is an associative and unital
action and these constructions yield mutually inverse bijections.
\end{example}

Further examples of the situation in Definition \ref{def:multi_bimonad_module}
are obtained from the following.

\begin{definition}\label{def:reg_mbm_module}
A {\em module} over a regular multiplier bimonoid $(t_1,t_2,t_3,t_4:A^2\to
A^2,e:A\to I)$ in a braided monoidal category $\C$ is an object 
$Q$ of $\C$ equipped with morphisms $q_1:A.Q\to A.Q$ and $q_4:Q.A\to Q.A$ in
$\C$ such that $(Q,q_1)$ is a module (in the sense of Definition
\ref{def:tric_module}) over the counital fusion morphism $t_1$ in
$\C$; $(Q,q_4)$ is a module over the counital fusion morphism 
$t_4$ in $\overline \C\rev$ and these structures are compatible in the sense
of the following diagrams.
$$
\xymatrix @R=12pt{
A^2.Q\ar[r]^-{1.q_1}\ar[dd]_-{t_4.1}&
A^2.Q\ar[dd]^-{t_4.1} &
Q.A^2\ar[r]^-{q_4.1} \ar[dd]_-{1.t_1}&
Q.A^2\ar[dd]^-{1.t_1} &
A.Q\ar[r]^-{q_1}\ar[d]_-{b}&
A.Q\ar[dd]^-{e.1}\\
&&&&Q.A\ar[d]_-{q_4}\\
A^2.Q\ar[r]_-{1.q_1}&
A^2.Q &
Q.A^2\ar[r]_-{q_4.1}&
Q.A^2&
Q.A\ar[r]_-{1.e}&
Q}
$$
(So that the common diagonal of the last diagram is a split epimorphism
by Definition \ref{def:tric_module}.)

A {\em morphism} of modules is a morphism $f:Q\to P$ in $\C$ such that the
following diagrams commute.
$$
\xymatrix{
A.Q\ar[r]^-{q_1}\ar[d]_-{1.f}&
A.Q\ar[d]^-{1.f} &
Q.A\ar[r]^-{q_4}\ar[d]_-{f.1}&
Q.A\ar[d]^-{f.1}\\
A.P\ar[r]_-{p_1}&
A.P &
P.A\ar[r]_-{p_4}&
P.A}
$$
\end{definition}

By the first and the last compatibility conditions on a module $(q_1,q_4)$ in
Definition~\ref{def:reg_mbm_module}, the diagram
$$
\xymatrix{
A^2.Q\ar[r]^-{1.q_1}\ar[d]_-{t_4.1}&
A^2.Q\ar[d]_-{t_4.1}\ar@/^2pc/[dd]^-{m b^{-1} .1} \\
A^2.Q\ar[r]^-{1.q_1}\ar[d]_-{1.q_4b}&
A^2.Q\ar[d]_-{1.e.1}\\
A.Q.A\ar[r]_-{1.1.e}&
A.Q}
$$
commutes. Hence if the multiplication $m$ is non-degenerate, then $q_1$ is
uniquely determined by $q_4$. Equivalently, using the second and the last
compatibility conditions in Definition~\ref{def:reg_mbm_module},
$q_4$ is uniquely determined by $q_1$. Furthermore, in this case, both
diagrams in Definition \ref{def:reg_mbm_module} defining morphisms of modules
become equivalent to each other: morphisms of modules can be defined by either
one of them.

\begin{example}\label{ex:mba_module}
Consider a regular multiplier bimonoid in $\mathsf{vec}$ induced by a regular
multiplier bialgebra over a field as in Example \ref{ex:reg_mbm_in_vec}. We
claim that its category of modules in Definition \ref{def:reg_mbm_module} is
isomorphic to the following category. The objects are vector spaces $Q$
equipped with an associative $A$-action $q:A.Q\to Q$ which is in addition a
surjective map. The morphisms are the linear maps which commute with the
actions.

The stated isomorphism acts on the morphisms as the identity map. It
takes a module $(q_1,q_4)$ in Definition \ref{def:reg_mbm_module} to the
associative and surjective action $(e.1)q_1=(1.e)q_4b$ (see Example
\ref{ex:tric_module_from_mba}), where $b$ stands for the symmetry in
$\mathsf{vec}$. In the opposite direction, we know from Example
\ref{ex:tric_module_from_mba} that associative and surjective actions
$q:A.Q\to Q$ are in a bijective correspondence with $t_1$-modules $q_1:A.Q\to
A.Q$, and also with $t_4$-modules $q_4:Q.A\to Q.A$, rendering commutative the
respective diagrams 
$$
\xymatrix{
A^2.Q\ar[r]^-{t_1.1}\ar[d]_-{1.q}&
A^2.Q\ar[d]^-{1.q} &
Q.A^2\ar[r]^-{1.t_4}\ar[d]_-{qb.1}&
Q.A^2\ar[d]^-{qb.1}\\
A.Q\ar@{-->}[r]_-{q_1}&
A.Q &
Q.A\ar@{-->}[r]_-{q_4}&
Q.A}
$$
We only need to show that $q_1$ and $q_4$ satisfy the compatibility conditions
in Definition \ref{def:reg_mbm_module}. The last one holds since the common
diagonal in the last diagram is the associative and surjective action
$q:A.Q\to Q$. Since $q$ is surjective, the first compatibility condition
follows by 
\begin{eqnarray*}
(t_4.1)(1.q_1)(1.1.q)&=&
(t_4.1)(1.1.q)(1.t_1.1)=
(1.1.q)(t_4.1.1)(1.t_1.1)\\
&=&(1.1.q)(1.t_1.1)(t_4.1.1)=
(1.q_1)(1.1.q)(t_4.1.1)=
(1.q_1)(t_4.1)(1.1.q).
\end{eqnarray*}
In the third equality we used axiom (B) in Definition \ref{def:reg_mbm}. The
second compatibility condition in Definition \ref{def:reg_mbm_module} follows
symmetrically. 
\end{example}

A module $(Q,q_1,q_4)$ in Definition \ref{def:reg_mbm_module} induces a module
over the induced multiplier bimonad $A.(-)$ in Example
\ref{ex:multi_bimonad_from_t} by putting
$$
\xymatrix@R=8pt{
\hat q:=\big(A.X.Q\ar[r]^-{1.b^{-1}}&
A.Q.X\ar[r]^-{q_1.1}&
A.Q.X\ar[r]^-{1.b}&
A.X.Q\big) \\
\check q:=\big(A.Q.X\ar[r]^-{b.1}&
Q.A.X\ar[r]^-{q_4.1}&
Q.A.X\big).}
$$
Hence from Theorem \ref{thm:mon_cat_of_multi_bimonad_modules} we obtain the
following.

\begin{corollary}
Consider a regular multiplier bimonoid $(t_1,t_2,t_3,t_4:A^2\to
A^2,e:A\to I)$ in a braided monoidal category $\C$. Assume that
$e$ and the equal morphisms (cf. the composite of axiom (B) in Definition
\ref{def:reg_mbm} with $1.e.1$)
$$
\xymatrix{
\big(A^3\ar[r]^-{1.t_1}&
A^3\ar[r]^-{b^{-1}.1}&
A^3\ar[r]^-{m.1}&
A^2\big)=\big(
A^3\ar[r]^-{t_4.1}&
A^3\ar[r]^-{1.m}&
A^2\big)}
$$
are split epimorphisms. Then the category of modules of
Definition~\ref{def:reg_mbm_module} is monoidal in such a way that the
forgetful functor to $\C$ is strict monoidal.
\end{corollary}

The monoidal product of modules $P$ and $Q$ is again a module via the
morphisms
$$
\xymatrix@R=8pt{
A.P.Q\ar[r]^-{1.b^{-1}}&
A.Q.P\ar[r]^-{q_1.1}&
A.Q.P\ar[r]^-{1.b}&
A.P.Q\ar[r]^-{p_1.1}&
A.P.Q\\
P.Q.A\ar[r]^-{1.b^{-1}}&
P.A.Q\ar[r]^-{p_4.1}&
P.A.Q\ar[r]^-{1.b}&
P.Q.A\ar[r]^-{1.q_4}&
P.Q.A.}
$$

\end{document}